\documentclass[a4paper,11pt]{article}
\usepackage{amsmath,amsthm,amssymb,enumitem,xcolor}

\usepackage[nosort,nocompress]{cite}

\usepackage[bookmarks=false,hyperfootnotes=false,colorlinks,
    linkcolor={red!60!black},
    citecolor={blue!50!black},
    urlcolor={blue!80!black}]{hyperref}

\renewcommand{\eqref}[1]{\hyperref[#1]{(\ref{#1})}}

\newlist{enumlist}{enumerate}{1}
\setlist[enumlist]{labelindent=0cm,label=\arabic*.,ref=\arabic*,labelwidth=2.5ex,labelsep=0.5ex,leftmargin=3ex,align=left,topsep=0.5ex,itemsep=1ex,parsep=1ex}

\newlist{itemlist}{itemize}{1}
\setlist[itemlist]{labelindent=0cm,label=$\bullet$,labelwidth=2.5ex,labelsep=0.5ex,leftmargin=3ex,align=left,topsep=0.5ex,itemsep=1ex,parsep=1ex}

\newlist{lijst}{itemize}{1}
\setlist[lijst]{labelindent=0cm,label={},labelwidth=4ex,labelsep=1ex,leftmargin=5ex,align=left,topsep=0.5ex,itemsep=1ex,parsep=1ex}

\numberwithin{equation}{section}

{\theoremstyle{definition}\newtheorem{definition}{Definition}[section]

\newtheorem{remark}[definition]{Remark}
\newtheorem{example}[definition]{Example}}

\newtheorem{proposition}[definition]{Proposition}
\newtheorem{lemma}[definition]{Lemma}
\newtheorem{theorem}[definition]{Theorem}
\newtheorem{corollary}[definition]{Corollary}

\newcommand{\C}{\mathbb{C}}
\newcommand{\cC}{\mathcal{C}}
\newcommand{\eps}{\varepsilon}
\newcommand{\al}{\alpha}
\newcommand{\be}{\beta}
\newcommand{\albar}{\overline{\alpha}}
\newcommand{\End}{\operatorname{End}}
\newcommand{\Irr}{\operatorname{Irr}}

\newcommand{\Rep}{\operatorname{Rep}}
\newcommand{\ot}{\otimes}
\newcommand{\recht}{\rightarrow}
\newcommand{\cb}{_\text{\rm cb}}
\newcommand{\Z}{\mathbb{Z}}
\newcommand{\vphi}{\varphi}

\newcommand{\cO}{\mathcal{O}}
\newcommand{\id}{\mathord{\text{\rm id}}}
\newcommand{\om}{\omega}
\newcommand{\SU}{\operatorname{SU}}
\newcommand{\bebar}{\overline{\beta}}
\newcommand{\pibar}{\overline{\pi}}

\newcommand{\cL}{\mathcal{L}}
\newcommand{\ovt}{\mathbin{\overline{\otimes}}}

\newcommand{\Tr}{\operatorname{Tr}}

\newcommand{\Om}{\Omega}

\newcommand{\R}{\mathbb{R}}

\newcommand{\counit}{\epsilon}

\newcommand{\cH}{\mathcal{H}}

\newcommand{\mubar}{\overline{\mu}}

\newcommand{\cZ}{\mathcal{Z}}
\newcommand{\Ad}{\operatorname{Ad}}

\newcommand{\gammabar}{\overline{\gamma}}
\newcommand{\cG}{\mathcal{G}}

\newcommand{\cK}{\mathcal{K}}

\newcommand{\cF}{\mathcal{F}}

\newcommand{\actson}{\curvearrowright}

\newcommand{\cS}{\mathcal{S}}
\newcommand{\cA}{\mathcal{A}}

\newcommand{\cHbar}{\overline{\cH}}
\newcommand{\cW}{\mathcal{W}}
\newcommand{\cU}{\mathcal{U}}

\newcommand{\cM}{\mathcal{M}}
\newcommand{\lspan}{\operatorname{span}}

\newcommand{\onb}{\operatorname{onb}}
\newcommand{\dpr}{^{\prime\prime}}

\newcommand{\cV}{\mathcal{V}}

\newcommand{\cE}{\mathcal{E}}

\newcommand{\Aut}{\operatorname{Aut}}

\newcommand{\vphitil}{\widetilde{\varphi}}

\newcommand{\cP}{\mathcal{P}}
\newcommand{\Bimod}{\operatorname{Bimod}}
\newcommand{\Stab}{\operatorname{Stab}}

\newcommand{\xitil}{\widetilde{\xi}}
\newcommand{\Pol}{\operatorname{Pol}}
\newcommand{\Deltatil}{\widetilde{\Delta}}
\newcommand{\Vtil}{\widetilde{V}}
\newcommand{\betabar}{\overline{\beta}}
\newcommand{\Ztil}{\widetilde{Z}}
\newcommand{\mutil}{\widetilde{\mu}}
\newcommand{\Deltah}{\widehat{\Delta}}
\newcommand{\pitil}{\widetilde{\pi}}
\newcommand{\cQ}{\mathcal{Q}}
\newcommand{\bI}{\mathbb{I}}
\newcommand{\SL}{\operatorname{SL}}
\newcommand{\Q}{\mathbb{Q}}

\begin{document}

\begin{center}
{\boldmath\LARGE\bf C$^*$-tensor categories and subfactors \vspace{0.5ex}\\ for totally disconnected groups}

\bigskip

{\sc by Yuki Arano\footnote{Graduate School of Mathematics, University of Tokyo, 3-8-1, Komaba, Meguro-ku,
Tokyo (Japan), arano@ms.u-tokyo.ac.jp. Supported by the Research Fellow of the Japan Society for the Promotion of Science and the Program for Leading Graduate
Schools, MEXT, Japan.} and Stefaan Vaes\footnote{KU~Leuven, Department of Mathematics, Leuven (Belgium), stefaan.vaes@wis.kuleuven.be \\
    Supported in part by European Research Council Consolidator Grant 614195, and by long term structural funding~-- Methusalem grant of the Flemish Government.}}
\end{center}

\begin{abstract}\noindent
We associate a rigid C$^*$-tensor category $\cC$ to a totally disconnected locally compact group $G$ and a compact open subgroup $K < G$. We characterize when $\cC$ has the Haagerup property or property~(T), and when $\cC$ is weakly amenable. When $G$ is compactly generated, we prove that $\cC$ is essentially equivalent to the planar algebra associated by Jones and Burstein to a group acting on a locally finite bipartite graph. We then concretely realize $\cC$ as the category of bimodules generated by a hyperfinite subfactor.
\end{abstract}

\section{Introduction}

Rigid C$^*$-tensor categories arise as representation categories of compact groups and compact quantum groups and also as (part of) the standard invariant of a finite index subfactor. They can be viewed as a discrete group like structure and this analogy has lead to a lot of recent results with a flavor of geometric group theory, see \cite{PV14,NY15a,GJ15,NY15b,PSV15}.

In this paper, we introduce a rigid C$^*$-tensor category $\cC$ canonically associated with a totally disconnected locally compact group $G$ and a compact open subgroup $K < G$. Up to Morita equivalence, $\cC$ does not depend on the choice of $K$. The tensor category $\cC$ can be described in several equivalent ways, see Section \ref{sec.our-categories}. Here, we mention that the representation category of $K$ is a full subcategory of $\cC$ and that the ``quotient'' of the fusion algebra of $\cC$ by $\Rep K$ is the Hecke algebra of finitely supported functions on $K \backslash G / K$ equipped with the convolution product.

When $G$ is compactly generated, we explain how the C$^*$-tensor category $\cC$ is related to the planar algebra $\cP$ (i.e.\ standard invariant of a subfactor) associated in \cite{Jo98,Bu10} with a locally finite bipartite graph $\cG$ and a closed subgroup $G < \Aut(\cG)$. At the same time, we prove that these planar algebras $\cP$ can be realized by a \emph{hyperfinite} subfactor.

Given a finite index subfactor $N \subset M$, the notions of \emph{amenability}, \emph{Haagerup property} and \emph{property~(T)} for its standard invariant $\cG_{N,M}$ were introduced by Popa in \cite{Po94,Po99,Po01} in terms of the associated symmetric enveloping algebra $T \subset S$ (see \cite{Po94,Po99}) and shown to only depend on $\cG_{N,M}$. Denoting by $\cC$ the tensor category of $M$-$M$-bimodules generated by the subfactor, these properties were then formulated in \cite{PV14} intrinsically in terms of $\cC$, and in particular directly in terms of $\cG_{N,M}$. We recall these definitions and equivalent formulations in Section~\ref{sec.haagerup-T}. Similarly, \emph{weak amenability} and the corresponding Cowling-Haagerup constant for the standard invariant $\cG_{N,M}$ of a subfactor $N \subset M$ were first defined in terms of the symmetric enveloping inclusion in \cite{Br14} and then intrinsically for rigid C$^*$-tensor categories in \cite{PV14}, see Section~\ref{sec.weak-amen-tensor-cat}. Reinterpreting \cite{DFY13,Ar14}, it was proved in \cite{PV14} that the representation category of $\SU_q(2)$ (and thus, the Temperley-Lieb-Jones standard invariant) is weakly amenable and has the Haagerup property, while the representation category of $\SU_q(3)$ has property~(T).

For the C$^*$-tensor categories $\cC$ that we associate to a totally disconnected group $G$, we characterize when $\cC$ has the Haagerup property or property~(T) and when $\cC$ is weakly amenable. We give several examples and counterexamples, in particular illustrating that the Haagerup property/weak amenability of $G$ is not sufficient for $\cC$ to have the Haagerup property or to be weakly amenable. Even more so, when $\cC$ is the category associated with $G = \SL(2,\Q_p)$, then the subcategory $\Rep K$ with $K = \SL(2,\Z_p)$ has the relative property~(T). When $G = \SL(n,\Q_p)$ with $n \geq 3$, the tensor category $\cC$ has property~(T), but we also give examples of property~(T) groups $G$ such that $\cC$ does not have property~(T).

Our main technical tool is Ocneanu's tube algebra \cite{Oc93} associated with any rigid C$^*$-tensor category, see Section \ref{sec.tube-algebra}. When $\cC$ is the C$^*$-tensor category of a totally disconnected group $G$, we prove that the tube algebra is isomorphic with a canonical dense $*$-subalgebra of $C_0(G) \rtimes_{\Ad} G$, where $G$ acts on $G$ by conjugation. We can therefore express the above mentioned approximation and rigidity properties of the tensor category $\cC$ in terms of $G$ and the dynamics of the conjugation action $G \actson^{\Ad} G$.

In this paper, all locally compact groups are assumed to be second countable. We call totally disconnected group every second countable, locally compact, totally disconnected group.

\section{\boldmath C$^*$-tensor categories of totally disconnected groups}\label{sec.our-categories}

Throughout this section, fix a totally disconnected group $G$. For all compact open subgroups $K_1, K_2 < G$, we define

\begin{lijst}
\item[$\cC_1 :$] the category of $K_1$-$K_2$-$L^\infty(G)$-modules, i.e.\ Hilbert spaces $\cH$ equipped with commuting unitary representations $(\lambda(k_1))_{k_1 \in K_1}$ and $(\rho(k_2))_{k_2 \in K_2}$ and with a normal $*$-representation $\Pi : L^\infty(G) \recht B(\cH)$ that are equivariant with respect to the left translation action $K_1 \actson G$ and the right translation action $K_2 \actson G$~;

\item[$\cC_2 :$] the category of $K_1$-$L^\infty(G/K_2)$-modules, i.e.\ Hilbert spaces $\cH$ equipped with a unitary representation $(\pi(k_1))_{k_1 \in K_1}$ and a normal $*$-representation $\Pi : L^\infty(G/K_2) \recht B(\cH)$ that are covariant with respect to the left translation action $K_1 \actson G/K_2$~;

\item[$\cC_3 :$] the category of $G$-$L^\infty(G/K_1)$-$L^\infty(G/K_2)$-modules, i.e.\ Hilbert spaces $\cH$ equipped with a unitary representation $(\pi(g))_{g \in G}$ and with an $L^\infty(G/K_1)$-$L^\infty(G/K_2)$-bimodule structure that are equivariant with respect to the left translation action of $G$ on $G/K_1$ and $G/K_2$~;
\end{lijst}
and with morphisms given by bounded operators that intertwine the given structure.

Let $K_3 < G$ also be a compact open subgroup. We define the tensor product $\cH \ot_{K_2} \cK$ of a $K_1$-$K_2$-$L^\infty(G)$-module $\cH$ and a $K_2$-$K_3$-$L^\infty(G)$-module $\cK$ as the Hilbert space
$$\cH \ot_{K_2} \cK = \{ \xi \in \cH \ot \cK \mid (\rho(k_2) \ot \lambda(k_2))\xi = \xi \;\;\text{for all}\;\; k_2 \in K_2 \}$$
equipped with the unitary representations $(\lambda(k_1) \ot 1)_{k_1 \in K_1}$ and $(1 \ot \rho(k_3))_{k_3 \in K_3}$ and with the representation $(\Pi_\cH \ot \Pi_\cK) \circ \Delta$ of $L^\infty(G)$, where $\Delta : L^\infty(G) \recht L^\infty(G) \ovt L^\infty(G)$ is the comultiplication given by $(\Delta(F))(g,h) = F(gh)$ for all $g,h \in G$.

The tensor product of a $G$-$L^\infty(G/K_1)$-$L^\infty(G/K_2)$-module $\cH$ and a $G$-$L^\infty(G/K_2)$-$L^\infty(G/K_3)$-module $\cK$ is denoted as $\cH \ot_{L^\infty(G/K_2)} \cK$ and defined as the Hilbert space
\begin{align*}
\cH \ot_{L^\infty(G/K_2)} \cK &= \{ \xi \in \cH \ot \cK \mid \xi (1_{g K_2} \ot 1) = (1 \ot 1_{gK_2}) \xi \;\;\text{for all}\;\; gK_2 \in G/K_2 \} \\
&= \bigoplus_{g \in G/K_2} \cH \cdot 1_{gK_2} \ot 1_{gK_2} \cdot \cK
\end{align*}
with the unitary representation $(\pi_\cH(g) \ot \pi_\cK(g))_{g \in G}$ and with the $L^\infty(G/K_1)$-$L^\infty(G/K_3)$-bimodule structure given by the left action of $1_{gK_1} \ot 1$ for $gK_1 \in G/K_1$ and the right action of $1 \ot 1_{h K_3}$ for $hK_3 \in G/K_3$.

We say that objects $\cH$ are of \emph{finite rank}
\begin{lijst}
\item[$\cC_1 :$] if $\cH_{K_2} := \{\xi \in \cH \mid \rho(k_2) \xi = \xi \;\;\text{for all}\;\; k_2 \in K_2 \}$ is finite dimensional~; as we will see in the proof of Proposition \ref{prop.some-equivalences-of-categories}, this is equivalent with requiring that $\ _{K_1}\cH$ is finite dimensional~;

\item[$\cC_2 :$] if $\cH$ is finite dimensional~;

\item[$\cC_3 :$] if $1_{eK_1} \cdot \cH$ is finite dimensional~; as we will see in the proof of Proposition \ref{prop.some-equivalences-of-categories}, this is equivalent with requiring that $\cH \cdot 1_{eK_2}$ is finite dimensional.
\end{lijst}

Altogether, we get that $\cC_1$ and $\cC_3$ are \emph{C$^*$-$2$-categories}. In both cases, the $0$-cells are the compact open subgroups of $G$. For all compact open subgroups $K_1, K_2 < G$, the $1$-cells are the categories $\cC_i(K_1,K_2)$ defined above and $\cC_i(K_1,K_2) \times \cC_i(K_2,K_3) \recht \cC_i(K_1,K_3)$ is given by the tensor product operation that we just introduced. Restricting to finite rank objects, we get rigid C$^*$-$2$-categories.

Another typical example of a C$^*$-$2$-category is given by Hilbert bimodules over II$_1$ factors: the $0$-cells are II$_1$ factors, the $1$-cells are the categories $\Bimod_{M_1\text{-}M_2}$ of Hilbert $M_1$-$M_2$-bimodules and $\Bimod_{M_1\text{-}M_2} \times \Bimod_{M_2\text{-}M_3} \recht \Bimod_{M_1\text{-}M_3}$ is given by the Connes tensor product. Again, restricting to finite index bimodules, we get a rigid C$^*$-$2$-category.

\begin{remark}\label{rem.2-category-subfactor}
The standard invariant of an extremal finite index subfactor $N \subset M$ can be viewed as follows as a rigid C$^*$-$2$-category. There are only two $0$-cells, namely $N$ and $M$; the $1$-cells are
the $N$-$N$, $N$-$M$, $M$-$N$ and $M$-$M$-bimodules generated by the subfactor; and we are given a favorite and generating $1$-cell from $N$ to $M$, namely the $N$-$M$-bimodule $L^2(M)$.

Abstractly, a rigid C$^*$-$2$-category $\cC$ with only two $0$-cells (say $+$ and $-$), irreducible tensor units in $\cC_{++}$ and $\cC_{--}$, and a given generating object $\cH \in \cC_{+-}$ is exactly the same as a standard $\lambda$-lattice in the sense of Popa \cite[Definitions 1.1 and 2.1]{Po94b}. Indeed, for every $n \geq 0$, define $\cH_{+,n}$ as the $n$-fold alternating tensor product of $\cH$ and $\cHbar$ starting with $\cH$. Similarly, define $\cH_{-,n}$ by starting with $\cHbar$. For $0 \leq j$, define $A_{0j} = \End(\cH_{+,j})$. When $0 \leq i \leq j < \infty$, define $A_{ij} \subset A_{0j}$ as $A_{ij} := 1^{i} \ot \End(\cH_{(-1)^i,j-i})$ viewed as a subalgebra of $A_{0j} = \End(\cH_{+,j})$ by writing $\cH_{+,j} = \cH_{+,i} \cH_{(-1)^i,j-i}$. The standard solutions for the conjugate equations (see Section \ref{sec.tube-algebra}) give rise to canonical projections $e_+ \in \End(\cH \cHbar)$ and $e_- \in \End(\cHbar \cH)$ given by
$$e_+ = d(\cH)^{-1} s_\cH s_\cH^* \quad\text{and}\quad e_- = d(\cH)^{-1} t_\cH t_\cH^* \; ,$$
and thus to a representation of the Jones projections $e_j \in A_{kl}$ (for $k < j < l$). Finally, if we equip all $A_{ij}$ with the normalized categorical trace, we have defined a standard $\lambda$-lattice in the sense of \cite[Definitions 1.1 and 2.1]{Po94b}. Given two rigid C$^*$-$2$-categories with fixed generating objects as above, it is straightforward to check that the associated standard $\lambda$-lattices are isomorphic if and only if there exists an equivalence of C$^*$-$2$-categories preserving the generators. Conversely given a standard $\lambda$-lattice $\cG$, by \cite[Theorem 3.1]{Po94b}, there exists an extremal subfactor $N \subset M$ whose standard invariant is $\cG$ and we can define $\cC$ as the C$^*$-$2$-category of the subfactor $N \subset M$, generated by the $N$-$M$-bimodule $L^2(M)$ as in the beginning of this remark. One can also define $\cC$ directly in terms of $\cG$ (see e.g.\ \cite[Section 4.1]{MPS08} for a planar algebra version of this construction).

Thus, also subfactor planar algebras in the sense of \cite{Jo99} are ``the same'' as rigid C$^*$-$2$-categories with two $0$-cells and such a given generating object $\cH \in \cC_{+-}$.
\end{remark}


For more background on rigid C$^*$-tensor categories, we refer to \cite{NT13}.

\begin{proposition}\label{prop.some-equivalences-of-categories}
The C$^*$-$2$-categories $\cC_1$ and $\cC_3$ are naturally equivalent. In particular, fixing $K_1=K_2=K$, we get the naturally equivalent rigid C$^*$-tensor categories $\cC_{1,f}(K < G)$ and $\cC_{3,f}(K < G)$. Up to Morita equivalence\footnote{In the sense of \cite[Section 4]{Mu01}, where the terminology weak Morita equivalence is used; see also \cite[Definition 7.3]{PSV15} and \cite[Section 3]{NY15b}.}, these do not depend on the choice of compact open subgroup $K < G$.
\end{proposition}
\begin{proof}
Using the left and right translation operators $\lambda_g$ and $\rho_g$ on $L^2(G)$, one checks that the following formulae define natural equivalences and their inverses between the categories $\cC_1$, $\cC_2$ and $\cC_3$.

\begin{itemlist}
\item $\cC_1 \recht \cC_2 : \cH \mapsto \cH_{K_2}$, where $\cH_{K_2}$ is the space of right $K_2$-invariant vectors and where the $K_1$-$L^\infty(G/K_2)$-module structure on $\cH_{K_2}$ is given by restricting the corresponding structure on $\cH$.

\item $\cC_2 \recht \cC_1 : \cH \mapsto \cH \ot_{L^\infty(G/K_2)} L^2(G)$ given by
\begin{align*}
& \{\xi \in \cH \ot L^2(G) \mid (1_{g K_2} \ot 1) \xi = (1 \ot 1_{gK_2})\xi \;\;\text{for all}\;\; g \in G \} \\
& \hspace{7cm} = \bigoplus_{g \in G/K_2} 1_{gK_2} \cdot \cH \ot L^2(g K_2)
\end{align*}
and where the $K_1$-$K_2$-$L^\infty(G)$-module structure is given by $(\lambda_\cH(k_1) \ot \lambda_{k_1})_{k_1 \in K_1}$, $(1 \ot \rho_{k_2})_{k_2 \in K_2}$ and multiplication with $1 \ot F$ when $F \in L^\infty(G)$.

\item $\cC_3 \recht \cC_2 : \cH \mapsto 1_{eK_1} \cdot \cH$ and where the $K_1$-$L^\infty(G/K_2)$-module structure on $1_{eK_1} \cdot \cH$ is given by restricting the corresponding structure on $\cH$.

\item $\cC_2 \recht \cC_3 : \cH \mapsto L^2(G) \ot_{K_1} \cH = \{ \xi \in L^2(G) \ot \cH \mid (\rho_{k_1} \ot \pi(k_1))\xi = \xi \;\;\text{for all}\;\; k_1 \in K_1\}$ and where the $G$-$L^\infty(G/K_1)$-$L^\infty(G/K_2)$-module structure is given by $(\lambda_g \ot 1)_{g \in G}$, multiplication with $F \ot 1$ for $F \in L^\infty(G/K_1)$ and multiplication with $(\id \ot \Pi)\Delta(F)$ for $F \in L^\infty(G/K_2)$.
\end{itemlist}

By definition, if $\cH \in \cC_1$ has finite rank, the Hilbert space $\cH_{K_2}$ is finite dimensional. Conversely, if $\cK \in \cC_2$ and $\cK$ is a finite dimensional Hilbert space, then the corresponding object $\cH \in \cC_1$ has the property that both $\ _{K_1}\cH$ and $\cH_{K_2}$ are finite dimensional. Therefore, $\cH \in \cC_1$ has finite rank if and only if $\ _{K_1}\cH$ is a finite dimensional Hilbert space. A similar reasoning holds for objects in $\cC_3$.

It is straightforward to check that the resulting equivalence $\cC_1 \leftrightarrow \cC_3$ preserves tensor products, so that we have indeed an equivalence between the C$^*$-$2$-categories $\cC_1$ and $\cC_3$.

To prove the final statement in the proposition, it suffices to observe that for all compact open subgroups $K_1,K_2 < G$, we have that $L^2(K_1 K_2)$ is a nonzero finite rank $K_1$-$K_2$-$L^\infty(G)$-module and that $L^2(G/(K_1 \cap K_2))$ is a nonzero finite rank $G$-$L^\infty(G/K_1)$-$L^\infty(G/K_2)$-module, so that $\cC_{i,f}(K_1 < G)$ and $\cC_{i,f}(K_2 < G)$ are Morita equivalent for $i=1,3$.
\end{proof}

The rigid C$^*$-$2$-categories $\cC_1$ and $\cC_2$ can as follows be fully faithfully embedded in the category of bimodules over the hyperfinite II$_1$ factor. We construct this embedding in an extremal way in the sense of subfactors (cf.\ Corollary \ref{cor.subfactor-realization}).

To do so, given a totally disconnected group $G$, we fix a continuous action $G \actson^\al P$ of $G$ on the hyperfinite II$_\infty$ factor $P$ that is \emph{strictly outer} in the sense of \cite[Definition 2.1]{Va03}: the relative commutant $P' \cap P \rtimes G$ equals $\C 1$. Moreover, we should choose this action in such a way that $\Tr \circ \al_g = \Delta(g)^{-1/2} \Tr$ for all $g \in G$ (where $\Delta$ is the modular function on $G$) and such that there exists a projection $p \in P$ of finite trace with the property that $\al_k(p)=p$ whenever $k$ belongs to a compact subgroup of $G$.
Such an action indeed exists: write $P = R_0 \ovt R_1$ where $R_0$ is a copy of the hyperfinite II$_1$ factor and $R_1$ is a copy of the hyperfinite II$_\infty$ factor. Choose a continuous trace scaling action $\R_0^+ \actson^{\al_1} R_1$. By \cite[Corollary 5.2]{Va03}, we can choose a strictly outer action $G \actson^{\al_0} R_0$. We then define $\al_g = (\al_0)_g \ot (\al_1)_{\Delta(g)^{-1/2}}$ and we take $p = 1 \ot p_1$, where $p_1 \in R_1$ is any projection of finite trace. Whenever $k$ belongs to a compact subgroup of $G$, we have $\Delta(k)=1$ and thus $\al_k(p) = p$.

Whenever $K_1,K_2 < G$ are compact open subgroups of $G$, we write
$$[K_1:K_2] = [K_1 : K_1 \cap K_2] \, [K_2 : K_1 \cap K_2]^{-1} \; .$$
Fixing a left Haar measure $\lambda$ on $G$, we have $[K_1:K_2] = \lambda(K_1) \, \lambda(K_2)^{-1}$. Therefore, we have that $[K : gKg^{-1}] = \Delta(g)$ for all compact open subgroups $K < G$ and all $g \in G$.

\begin{theorem}\label{thm.functor-to-bimod}
Let $G$ be a totally disconnected group and choose a strictly outer action $G \actson^\al P$ on the hyperfinite II$_\infty$ factor $P$ and a projection $p \in P$ as above. For every compact open subgroup $K < G$, write $R(K) = (pPp)^K$. Then each $R(K)$ is a copy of the hyperfinite II$_1$ factor.

To every $K_1$-$K_2$-$L^\infty(G)$-module $\cH$, we associate the Hilbert $R(K_1)$-$R(K_2)$-bimodule $\cK$ given by \eqref{eq.bim-bim} below. Then $\cH \mapsto \cK$ is a fully faithful $2$-functor. Also, $\cH$ has finite rank if and only if $\cK$ is a finite index bimodule. In that case,
$$\dim_{R(K_1)-}(\cK) = [K_1:K_2]^{1/2} \; \dim_{\cC_1}(\cH) \quad\text{and}\quad \dim_{-R(K_2)}(\cK) = [K_2:K_1]^{1/2} \; \dim_{\cC_1}(\cH) \; ,$$
where $\dim_{\cC_1}(\cH)$ is the categorical dimension of $\cH \in \cC_1$.
\end{theorem}
\begin{proof}
Given a $K_1$-$K_2$-$L^\infty(G)$-module $\cH$, turn $\cH \ot L^2(P)$ into a Hilbert $(P \rtimes K_1)$-$(P \rtimes K_2)$-bimodule via
\begin{align*}
& u_k \cdot (\xi \ot b) \cdot u_r = \lambda(k) \rho(r)^* \xi \ot \al_r^{-1}(b) & &\text{for all}\;\; k \in K_1 , r \in K_2 , \xi \in \cH , b \in L^2(P) \; ,\\
& a \cdot \zeta \cdot d = (\Pi \ot \id)\al(a) \, \zeta \, (1 \ot d) & &\text{for all}\;\; a,d \in P, \zeta \in \cH \ot L^2(P) \; ,
\end{align*}
where $\al : P \recht L^\infty(G) \ovt P$ is given by $(\al(a))(g) = \al_g^{-1}(a)$.

Whenever $K < G$ is a compact open subgroup, we define the projection $p_K \in L(G)$ given by
$$p_K = \lambda(K)^{-1} \int_K \lambda_k \, dk \; .$$
We also write $e_K = p p_K$ viewed as a projection in $P \rtimes K$. Since $P \subset P \rtimes K \subset P \rtimes G$, we have that $P' \cap (P \rtimes K) = \C 1$, so that $P \rtimes K$ is a factor. So, $P \rtimes K$ is a copy of the hyperfinite II$_\infty$ factor and $e_K \in P \rtimes K$ is a projection of finite trace. We identify $R(K) = e_K (P \rtimes K) e_K$ through the bijective $*$-isomorphism $(pPp)^K \recht e_K (P \rtimes K) e_K : a \mapsto a p_K$. In particular, $R(K)$ is a copy of the hyperfinite II$_1$ factor.

So, for every $K_1$-$K_2$-$L^\infty(G)$-module $\cH$, we can define the $R(K_1)$-$R(K_2)$-bimodule
\begin{equation}\label{eq.bim-bim}
\cK = e_{K_1} \cdot (\cH \ot L^2(P)) \cdot e_{K_2} \; .
\end{equation}

We claim that $\End_{R(K_1)-R(K_2)}(\cK) = \End_{\cC_1}(\cH)$ naturally. More concretely, we have to prove that
\begin{equation}\label{eq.to-prove}
\End_{(P \rtimes K_1)-(P \rtimes K_2)}(\cH \ot L^2(P)) = \End_{\cC_1}(\cH) \ot 1 \; ,
\end{equation}
where $\End_{\cC_1}(\cH)$ consists of all bounded operators on $\cH$ that commute with $\lambda(K_1)$, $\rho(K_2)$ and $\Pi(L^\infty(G))$. To prove \eqref{eq.to-prove}, it is sufficient to show that
\begin{equation}\label{eq.suff}
\End_{P-P}(\cH \ot L^2(P)) = \Pi(L^\infty(G))' \ot 1 \; .
\end{equation}
Note that the left hand side of \eqref{eq.suff} equals $(\Pi \ot \id)\al(P)' \cap B(\cH) \ovt P$. Assume that $T \in (\Pi \ot \id)\al(P)' \cap B(\cH) \ovt P$.
In the same was as in \cite[Proposition 2.7]{Va03}, it follows that $T \in \Pi(L^\infty(G))' \cap 1$. For completeness, we provide a detailed argument. Define the unitary $W \in L^\infty(G) \ovt L(G)$ given by $W(g) = \lambda_g$. We view both $T$ and $(\Pi \ot \id)(W)$ as elements in $B(\cH) \ovt (P \rtimes G)$. For all $a \in P$, we have
\begin{align*}
(\Pi \ot \id)(W) \, T \, (\Pi \ot \id)(W)^* \, (1 \ot a) &= (\Pi \ot \id)(W) \, T \, (\Pi \ot \id)\al(a) \, (\Pi \ot \id)(W)^* \\ &= (1 \ot a) \, (\Pi \ot \id)(W) \, T \, (\Pi \ot \id)(W)^* \; .
\end{align*}
Since the action $\al$ is strictly outer, we conclude that $(\Pi \ot \id)(W) \, T \, (\Pi \ot \id)(W)^* = S \ot 1$ for some $S \in B(\cH)$. So,
$$T = (\Pi \ot \id)(W)^* \, (S \ot 1) \, (\Pi \ot \id)(W) \; .$$
The left hand side belongs to $B(\cH) \ovt P$, while the right hand side belongs to $B(\cH) \ot L(G)$, and both are viewed inside $B(\cH) \ovt (P \rtimes G)$. Since $P \cap L(G) = \C 1$, we conclude that $T = T_0 \ot 1$ for some $T_0 \in B(\cH)$ and that
$$T_0 \ot 1 = (\Pi \ot \id)(W)^* \, (S \ot 1) \, (\Pi \ot \id)(W) \; .$$
Defining the normal $*$-homomorphism $\Psi : L(G) \recht L(G) \ovt L(G)$ given by $\Psi(\lambda_g) = \lambda_g \ot \lambda_g$ for all $g \in G$, we apply $\id \ot \Psi$ and conclude that
\begin{align*}
T_0 \ot 1 \ot 1 &= (\Pi \ot \id)(W)_{13}^* \, (\Pi \ot \id)(W)_{12}^* \, (S \ot 1) \, (\Pi \ot \id)(W)_{12} \, (\Pi \ot \id)(W)_{13} \\ &= (\Pi \ot \id)(W)_{13}^* \, (T_0 \ot 1 \ot 1) \, (\Pi \ot \id)(W)_{13} \; .
\end{align*}
It follows that $T_0$ commutes with $\Pi(L^\infty(G))$ and \eqref{eq.to-prove} is proven.

It is easy to check that $\cH \mapsto \cK$ naturally preserves tensor products. So, we have found a fully faithful $2$-functor from $\cC_1$ to the C$^*$-$2$-category of Hilbert bimodules over hyperfinite II$_1$ factors.

To compute $\dim_{-R(K_2)}(\cK)$, observe that for all $k \in K_1$, $r \in K_2$ and $g \in G$, we have $\al_{k g r}(p) = \al_{kg}(p) = \al_g(\al_{g^{-1} k g}(p)) = \al_g(p)$. Therefore, as a right $(P \rtimes K_2)$-module, we have
$$e_{K_1} \cdot (\cH \ot L^2(P)) \cong \bigoplus_{g \in K_1 \backslash G / K_2} \bigl(\cL_g \ot L^2(p_g P)\bigr) \; ,$$
where $p_g = \al_g^{-1}(p)$, where the Hilbert space $\cL_g := \Pi(1_{K_1 g K_2})(\ _{K_1} \cH )$ comes with the unitary representation $(\rho(r))_{r \in K_2}$ and where the right $(P \rtimes K_2)$-module structure on $\cL_g \ot L^2(p_g P)$ is given by
$$(\xi \ot b) \cdot (d u_r) = \rho(r)^* \xi \ot \al_r^{-1}(bd) \quad\text{for all}\;\; \xi \in \cL_g , b \in L^2(p_g P), d \in P, r \in K_2 \; .$$

Since $p_g P p_g \rtimes K_2 = p_g(P \rtimes K_2)p_g$ is a factor (actually, $K_2 \actson p_g P p_g$ is a so-called minimal action), it follows from \cite[Theorem 12]{Wa88} that there exists a unitary $V_g \in B(\cL_g) \ovt p_g P p_g$ satisfying
$$(\id \ot \al_r)(V_g) = V_g (\rho(r) \ot 1) \quad \text{for all}\;\; r \in K_2 \; .$$
Then left multiplication with $V_g$ intertwines the right $(P \rtimes K_2)$-module structure on the Hilbert space $\cL_g \ot L^2(p_g P)$ with the right $(P \rtimes K_2)$-module structure given by
$$(\xi \ot b) \cdot (d u_r) = \xi \ot \al_r^{-1}(bd) \quad\text{for all}\;\; \xi \in \cL_g , b \in L^2(p_g P), d \in P, r \in K_2 \; .$$
Therefore,
\begin{align*}
\dim_{-R(K_2)} \bigl(\cL_g \ot L^2(p_g P)\bigr) \cdot e_{K_2} & = \dim(\cL_g) \, \dim_{-(pPp)^{K_2}} \bigl(L^2(p_g P^{K_2} p)\bigr) \\ &= \dim(\cL_g) \, \frac{\Tr(p_g)}{\Tr(p)} = \dim(\cL_g) \, \Delta(g)^{1/2} \; .
\end{align*}
So, we have proved that
$$\dim_{-R(K_2)} (\cK) = \sum_{g \in K_1 \backslash G / K_2} \dim\bigl( \Pi(1_{K_1 g K_2})(\ _{K_1} \cH ) \bigr) \; \Delta(g)^{1/2} \; .$$
We similarly get that
$$\dim_{R(K_1)-} (\cK) = \sum_{g \in K_1 \backslash G / K_2} \dim\bigl( \Pi(1_{K_1 g K_2})(\cH_{K_2} ) \bigr) \; \Delta(g)^{-1/2} \; .$$

To make the connection with the categorical dimension of $\cH$, it is useful to view $\cH$ as the image of a $G$-$L^\infty(G/K_1)$-$L^\infty(G/K_2)$-module $\cH'$ under the equivalence of Proposition \ref{prop.some-equivalences-of-categories}. This means that we can view $\cH$ as the space of $L^2$-functions $\xi : G \recht \cH'$ with the property that $\xi(g) \in 1_{eK_1} \cdot \cH' \cdot 1_{g K_2}$ for a.e.\ $g \in G$. The $L^\infty(G)$-module structure of $\cH$ is given by pointwise multiplication, while the $K_1$-$K_2$-module structure on $\cH$ is given by
$$(k \cdot \xi \cdot r)(g) = \pi(k) \xi(k^{-1} g r^{-1}) \quad\text{for all}\;\; k \in K_1, r \in K_2, g \in G \; .$$
With this picture, it is easy to see that
$$\Pi(1_{K_1 g K_2})(\cH_{K_2}) \cong 1_{eK_1} \cdot \cH' \cdot 1_{K_1 g K_2} \; .$$
The map $\xi \mapsto \xitil$ with $\xitil(g) = \pi(g)^* \xi(g)$ is an isomorphism between $\cH$ and the space of $L^2$-functions $\eta : G \recht \cH'$ with the property that $\eta(g) \in 1_{g^{-1}K_1} \cdot \cH' \cdot 1_{eK_2}$ for a.e.\ $g \in G$. The $L^\infty(G)$-module structure is still given by pointwise multiplication, while the $K_1$-$K_2$-module structure is now given by
$$(k \cdot \eta \cdot r)(g) = \pi(r)^* \eta(k^{-1} g r^{-1}) \; .$$
In this way, we get that
$$\Pi(1_{K_1 g K_2})(\ _{K_1} \cH ) \cong 1_{K_2 g^{-1} K_1} \cdot \cH' \cdot 1_{eK_2} \; .$$
It thus follows that
\begin{align}
\dim_{-R(K_2)}(\cK) &= \sum_{g \in K_1 \backslash G / K_2} \dim(1_{K_2 g^{-1} K_1} \cdot \cH' \cdot 1_{eK_2}) \, \Delta(g)^{1/2} \quad\text{and}\label{eq.dim-r-cK}\\
\dim_{R(K_1)-}(\cK) &= \sum_{g \in K_1 \backslash G / K_2} \dim(1_{eK_1} \cdot \cH' \cdot 1_{K_1 g K_2}) \, \Delta(g)^{-1/2} \; .\label{eq.dim-l-cK}
\end{align}

Also note that for every $g \in G$, we have
\begin{align*}
\dim(1_{K_2 g^{-1} K_1} \cdot \cH' \cdot 1_{eK_2}) &= [K_2 : K_2 \cap g^{-1} K_1 g] \; \dim(1_{g^{-1}K_1} \cdot \cH' \cdot 1_{eK_2}) \\
&= [K_2 : K_2 \cap g^{-1} K_1 g] \; \dim(1_{eK_1} \cdot \cH' \cdot 1_{gK_2}) \\
&= \frac{[K_2 : K_2 \cap g^{-1} K_1 g]}{[K_1 : K_1 \cap g K_2 g^{-1}]} \; \dim(1_{eK_1} \cdot \cH' \cdot 1_{K_1 gK_2}) \\
&= [K_2 : K_1] \; \Delta(g)^{-1} \; \dim(1_{eK_1} \cdot \cH' \cdot 1_{K_1 gK_2}) \; .
\end{align*}
It follows that
\begin{align*}
\dim_{-R(K_2)}(\cK) &= [K_2 : K_1] \; \sum_{g \in K_1 \backslash G / K_2} \dim(1_{eK_1} \cdot \cH' \cdot 1_{K_1 g K_2}) \, \Delta(g)^{-1/2} \\ &= [K_2 : K_1] \; \dim_{R(K_1)-}(\cK) \; .
\end{align*}

If $\cH$ has finite rank, also $\cH'$ has finite rank so that $\cH' \cdot 1_{eK_2}$ and $1_{eK_1} \cdot \cH'$ are finite dimensional Hilbert spaces. It then follows that $\cK$ is a finite index bimodule.

Conversely, assume that $\cK$ has finite index. For every $g \in G$, write
$$\kappa(g) := \dim(1_{K_2 g^{-1} K_1} \cdot \cH' \cdot 1_{eK_2}) \, \Delta(g)^{1/2} = [K_2 : K_1] \; \dim(1_{eK_1} \cdot \cH' \cdot 1_{K_1 g K_2}) \, \Delta(g)^{-1/2} \; .$$
So,
$$\kappa(g)^2 = [K_2 : K_1] \, \dim(1_{K_2 g^{-1} K_1} \cdot \cH' \cdot 1_{eK_2}) \, \dim(1_{eK_1} \cdot \cH' \cdot 1_{K_1 g K_2}) \; .$$
Thus, whenever $\kappa(g) \neq 0$, we have that $\kappa(g) \geq [K_2:K_1]^{1/2}$. Since
$$\dim_{-R(K_2)}(\cK) = \sum_{g \in K_1 \backslash G / K_2} \kappa(g) \; ,$$
we conclude that there are only finitely many $g \in K_1 \backslash G / K_2$ for which $1_{K_2 g^{-1} K_1} \cdot \cH' \cdot 1_{eK_2}$ is nonzero and for each of them, it is a finite dimensional Hilbert space. This implies that $\cH' \cdot 1_{eK_2}$ is finite dimensional, so that $\cH'$ has finite rank.

We have proved that $\cH \mapsto \cK$ is a fully faithful $2$-functor from $\cC_{1,f}$ to the finite index bimodules over hyperfinite II$_1$ factors. Moreover, for given compact open subgroups $K_1,K_2 < G$, the ratio between $\dim_{R(K_1)-}(\cK)$ and $\dim_{-R(K_2)}(\cK)$ equals $[K_1:K_2]$ for all finite rank $K_1$-$K_2$-$L^\infty(G)$-modules $\cH$. Since the functor is fully faithful, this then also holds for all $R(K_1)$-$R(K_2)$-subbimodules of $\cK$. It follows that the categorical dimension of $\cK$ equals
$$[K_2 : K_1]^{1/2} \, \dim_{R(K_1)-}(\cK) = [K_1:K_2]^{1/2} \, \dim_{-R(K_2)}(\cK) \; .$$
Since the functor is fully faithful, the categorical dimensions of $\cH \in \cC_{1,f}$ and $\cK \in \Bimod_f$ coincide, so that
\begin{equation}\label{eq.cat-dim-cK}
[K_2 : K_1]^{1/2} \, \dim_{R(K_1)-}(\cK) = \dim_{\cC_1}(\cH) = [K_1:K_2]^{1/2} \, \dim_{-R(K_2)}(\cK) \; .
\end{equation}
\end{proof}

\begin{corollary}\label{cor.subfactor-realization}
Let $G$ be a totally disconnected group with compact open subgroups $K_\pm < G$ and assume that $\cH$ is a finite rank $G$-$L^\infty(G/K_+)$-$L^\infty(G/K_-)$-module. Denote by $\cC = (\cC_{++},\cC_{+-},\cC_{-+},\cC_{--})$ the C$^*$-$2$-category of $G$-$L^\infty(G/K_\pm)$-$L^\infty(G/K_\pm)$-modules (with $0$-cells $K_+$ and $K_-$) generated by the alternating tensor products of $\cH$ and its adjoint.

Combining Proposition \ref{prop.some-equivalences-of-categories} and Theorem \ref{thm.functor-to-bimod}, we find an extremal hyperfinite subfactor $N \subset M$ whose standard invariant, viewed as the C$^*$-$2$-category of $N$-$N$, $N$-$M$, $M$-$N$ and $M$-$M$-bimodules generated by the $N$-$M$-bimodule $L^2(M)$, is equivalent with $(\cC,\cH)$ (cf.\ Remark \ref{rem.2-category-subfactor}).
\end{corollary}
\begin{proof}
A combination of Proposition \ref{prop.some-equivalences-of-categories} and Theorem \ref{thm.functor-to-bimod} provides the finite index $R(K_+)$-$R(K_-)$-bimodule $\cK$ associated with $\cH$. Take nonzero projections $p_\pm \in R(K_\pm)$ such that writing $N = p_+ R(K_+) p_+$ and $M = p_- R(K_-) p_-$, we have that $\dim_{-M}(p_+ \cdot \cK \cdot p_-) = 1$. We can then view $N \subset M$ in such a way that $L^2(M) \cong p_+ \cdot \cK \cdot p_-$ as $N$-$M$-bimodules. The C$^*$-$2$-category of $N$-$N$, $N$-$M$, $M$-$N$ and $M$-$M$-bimodules generated by the $N$-$M$-bimodule $L^2(M)$ is by construction equivalent with the rigid C$^*$-$2$-category of $R(K_\pm)$-$R(K_\pm)$-bimodules generated by $\cK$. Since the $2$-functor in Theorem \ref{thm.functor-to-bimod} is fully faithful, this C$^*$-$2$-category is equivalent with $\cC$ and this equivalence maps the $N$-$M$-bimodule $L^2(M)$ to $\cH \in \cC_{+-}$.
\end{proof}

From Corollary \ref{cor.subfactor-realization}, we get the following result.

\begin{proposition}\label{prop.link-to-jones-burstein}
Let $\cP$ be the subfactor planar algebra of \cite{Jo98,Bu10} associated with a connected locally finite bipartite graph $\cG$, with edge set $\cE$ and source and target maps $s : \cE \recht V_+$, $t : \cE \recht V_-$, together with\footnote{Note that in \cite{Bu10}, also a weight function $\mu : V_+ \sqcup V_- \recht \R_0^+$ scaled by the action of $G$ is part of the construction. But only when we take $\mu$ to be a multiple of the function $v \mapsto [\Stab v: \Stab v_+]^{1/2}$, we actually obtain a subfactor planar algebra, contrary to what is claimed in \cite[Proposition 4.1]{Bu10}.} a closed subgroup $G < \Aut(\cG)$ acting transitively on $V_+$ as well as on $V_-$. Fix vertices $v_\pm \in V_\pm$ and write $K_\pm = \Stab v_\pm$.

There exists an extremal hyperfinite subfactor $N \subset M$ whose standard invariant is isomorphic with $\cP$. We have $[M:N] = \delta^2$ where
\begin{align*}
\delta &= \sum_{w \in V_-} \# \{ e \in \cE \mid s(e) = v_+ , t(e) = w\} \, [\Stab w : \Stab v_+]^{1/2} \\ &= \sum_{w \in V_+} \# \{e \in \cE \mid s(e) = w , t(e) = v_-\} \, [\Stab w : \Stab v_-]^{1/2} \; .
\end{align*}
Moreover, $\cP$ can be described as the rigid C$^*$-$2$-category $\cC_{3,f}(G,K_\pm,K_\pm)$ of all finite rank $G$-$L^\infty(G/K_{\pm})$-$L^\infty(G/K_{\pm})$-modules together with the generating object $\ell^2(\cE) \in \cC_{3,f}(G,K_+,K_-)$ (cf.\ Remark \ref{rem.2-category-subfactor}).
\end{proposition}

\begin{proof}
We are given $G \actson \cE$ and $G \actson V_+$, $G \actson V_-$ such that the source and target maps $s,t$ are $G$-equivariant and such that $G$ acts transitively on $V_+$ and on $V_-$. Put $K_\pm = \Stab v_\pm$ and note that $K_\pm < G$ are compact open subgroups. We identify $G/K_\pm = V_\pm$ via the map $g K_\pm \mapsto g \cdot v_\pm$. In this way, $\cH := \ell^2(\cE)$ naturally becomes a finite rank $G$-$L^\infty(G/K_+)$-$L^\infty(G/K_-)$-module. Denote by $\cC$ the C$^*$-$2$-category of $G$-$L^\infty(G/K_\pm)$-$L^\infty(G/K_\pm)$-modules generated by the alternating tensor products of $\cH$ and its adjoint.

In the $2$-category $\cC_3$, the $n$-fold tensor product $\cH \ot \overline{\cH} \ot \cdots$ equals $\ell^2(\cE_{+,n})$, where $\cE_{+,n}$ is the set of paths in the graph $\cG$ starting at an even vertex and having length $n$. Similarly, the $n$-fold tensor product $\overline{\cH} \ot \cH \ot \cdots$ equals $\ell^2(\cE_{-,n})$, where $\cE_{-,n}$ is the set of paths of length $n$ starting at an odd vertex. So by construction, under the equivalence of Remark \ref{rem.2-category-subfactor}, $\cC$ together with its generator $\cH \in \cC_{+-}$ corresponds exactly to the planar algebra $\cP$ constructed in \cite{Bu10,Jo98}.

By Corollary \ref{cor.subfactor-realization}, we get that $(\cC,\cH)$ is the standard invariant of an extremal hyperfinite subfactor $N \subset M$. In particular, $[M:N] = \delta^2$ with $\delta = \dim_{\cC_3}(\cH)$. Combining \eqref{eq.cat-dim-cK} with \eqref{eq.dim-r-cK}, and using that
$$\Delta(g)^{-1/2} = [g K_+ g^{-1} : K_+]^{1/2} = [\Stab(g \cdot v_+) : K_+]^{1/2} \; ,$$
we get that
\begin{align*}
\delta &= [K_+ : K_-]^{1/2} \, \sum_{g \in G/K_+} \dim(1_{g K_+} \cdot \cH \cdot 1_{e K_-}) \, \Delta(g)^{-1/2} \\
&= \sum_{g \in G/K_+} \#\{e \in \cE \mid s(e) = g \cdot v_+ , t(e) = v_- \} \, [\Stab(g \cdot v_+) : K_+]^{1/2} \; [K_+ : K_-]^{1/2} \\
&= \sum_{w \in V_+} \# \{e \in \cE \mid s(e) = w , t(e) = v_- \} \, [\Stab w : \Stab v_-]^{1/2} \; .
\end{align*}
Combining \eqref{eq.cat-dim-cK} with \eqref{eq.dim-l-cK}, we similarly get that
$$\delta = \sum_{w \in V_-} \# \{e \in \cE \mid s(e) = v_+ , t(e) = w \} \, [\Stab w : \Stab v_+]^{1/2} \; .$$
To conclude the proof of the proposition, it remains to show that $\cC$ is equal to the C$^*$-$2$-category of all finite rank $G$-$L^\infty(G/K_\pm)$-$L^\infty(G/K_\pm)$-modules. For the $G$-$L^\infty(G/K_+)$-$L^\infty(G/K_-)$-mod\-ules, this amounts to proving that all irreducible representations of $K_+ \cap K_-$ appear in
$$\ell^2(\text{paths starting at $v_+$ and ending at $v_-$}) \; .$$
Since the graph is connected, the action of $K_+ \cap K_-$ on this set of paths is faithful and the result follows. The other cases are proved in the same way.
\end{proof}

\begin{remark}\label{rem.irreducible}
Note that the subfactors $N \subset M$ in Proposition \ref{prop.link-to-jones-burstein} are \emph{irreducible} precisely when $G$ acts transitively on the set of edges and there are no multiple edges. This means that the totally disconnected group $G$ is \emph{generated} by the compact open subgroups $K_\pm < G$ and that we can identify $\cE = G / (K_+ \cap K_-)$, $V_\pm = G/K_\pm$ with the natural source and target maps $G/(K_+ \cap K_-) \recht G/K_\pm$. The irreducible subfactor $N \subset M$ then has integer index given by $[M:N] = [K_+:K_+ \cap K_-] \, [K_-:K_+ \cap K_-]$.
\end{remark}

We finally note that the rigid C$^*$-tensor categories $\cC_{1,f}(K < G)$ and $\cC_{3,f}(K < G)$ also arise in a different way as categories of bimodules over a II$_1$ factor in the case where $K < G$ is the \emph{Schlichting completion} of a \emph{Hecke pair} $\Lambda < \Gamma$, cf.\ \cite[Section 4]{DV10}.

Recall that a Hecke pair consists of a countable group $\Gamma$ together with a subgroup $\Lambda < \Gamma$ that is almost normal, meaning that $g \Lambda g^{-1} \cap \Lambda$ has finite index in $\Lambda$ for all $g \in \Gamma$. The left translation action of $\Gamma$ on $\Gamma/\Lambda$ gives a homomorphism $\pi$ of $\Gamma$ to the group of permutations of $\Gamma/\Lambda$. The closure $G$ of $\pi(\Gamma)$ for the topology of pointwise convergence is a totally disconnected group and the stabilizer $K$ of the point $e\Lambda \in \Gamma/\Lambda$ is a compact open subgroup of $G$ with the property that $\Lambda = \pi^{-1}(K)$. One calls $(G,K)$ the Schlichting completion of the Hecke pair $(\Gamma,\Lambda)$. Note that there is a natural identification of $G/K$ and $\Gamma/\Lambda$.


\begin{proposition} \label{prop.identify-our-tensor-categories}
Let $\Lambda < \Gamma$ be a Hecke pair with Schlichting completion $K < G$. Choose an action $\Gamma \actson^\al P$ of $\Gamma$ by outer automorphisms of a II$_1$ factor $P$. Define $N = P \rtimes \Lambda$ and $M = P \rtimes \Gamma$. Note that $N \subset M$ is an irreducible, quasi-regular inclusion of II$_1$ factors. Denote by $\cC$ the tensor category of finite index $N$-$N$-bimodules generated by the finite index $N$-subbimodules of $L^2(M)$.

Then, $\cC$ and the earlier defined $\cC_{1,f}(K < G)$ and $\cC_{3,f}(K < G)$ are naturally equivalent rigid C$^*$-tensor categories.
\end{proposition}
\begin{proof}
Define
\begin{lijst}
\item[$\cC_4$~:] the category of $\Lambda$-$\Lambda$-$\ell^\infty(\Gamma)$-modules, i.e.\ Hilbert spaces $\cH$ equipped with two commuting unitary representations of $\Lambda$ and a representation of $\ell^\infty(\Gamma)$ that are covariant with respect to the left and right translation actions $\Lambda \actson \Gamma$~;

\item[$\cC_5$~:] the category of $\Lambda$-$\ell^\infty(\Gamma/\Lambda)$-modules, i.e.\ Hilbert spaces equipped with a unitary representation of $\Lambda$ and a representation of $\ell^\infty(\Gamma/\Lambda)$ that are covariant with respect to the left translation action $\Lambda \actson \Gamma/\Lambda$~:
\end{lijst}
with morphisms again given by bounded operators that intertwine the given structure.

To define the tensor product of two objects in $\cC_4$, it is useful to view $\cH \in \cC_4$ as a family of Hilbert spaces $(\cH_g)_{g \in \Gamma}$ together with unitary operators $\lambda(k) : \cH_g \recht \cH_{kg}$ and $\rho(k) : \cH_g \recht \cH_{g k^{-1}}$ for all $k \in \Lambda$, satisfying the obvious relations. The tensor product of two $\Lambda$-$\Lambda$-$\ell^\infty(\Gamma)$-modules $\cH$ and $\cK$ is then defined as
\begin{align*}
(\cH \ot_{\Lambda} \cK)_g = \Bigl\{ (\xi_h)_{h \in \Gamma} \Bigm | \;\; & \xi_h \in \cH_h \ot \cK_{h^{-1}g} \; , \\ & \xi_{hk^{-1}} = (\rho_\cH(k) \ot \lambda_\cK(k))(\xi_h) \; \text{for all}\; h \in \Gamma, k \in \Lambda \; , \\ & \sum_{h \in \Gamma/\Lambda} \|\xi_h\|^2 < \infty \Bigr\}
\end{align*}
with $\lambda(k) : (\cH \ot_{\Lambda} \cK)_g \recht (\cH \ot_{\Lambda} \cK)_{kg}$ given by $(\lambda(k)\xi)_h = (\lambda_\cH(k) \ot 1)\xi_{k^{-1}h}$ and $\rho(k) : (\cH \ot_{\Lambda} \cK)_g \recht (\cH \ot_{\Lambda} \cK)_{gk^{-1}}$ given by $(\rho(k)\xi)_h = (1 \ot \rho_\cK(k))\xi(h)$ for all $k \in \Lambda$, $h \in \Gamma$. Of course, choosing a section $i : \Gamma/\Lambda \recht \Gamma$, we have
$$(\cH \ot_\Lambda \cK)_g \cong \bigoplus_{h \in \Gamma/\Lambda} (\cH_{i(h)} \ot \cK_{i(h)^{-1} g}) \; ,$$
but this isomorphism depends on the choice of the section.

As in Proposition \ref{prop.some-equivalences-of-categories}, $\cC_4$ and $\cC_5$ are equivalent C$^*$-categories, where the equivalence and its inverse are defined as follows.
\begin{itemlist}
\item $\cC_4 \recht \cC_5 : \cH \mapsto \cK$, with
$$\cK_{g \Lambda} = \bigl\{ (\xi_h)_{h \in g \Lambda} \bigm | \xi_h \in \cH_h \; , \; \xi_{hk^{-1}} = \rho(k) \xi_h \; \text{for all}\; h \in g \Lambda, k \in \Lambda \bigr\}$$
and with the natural $\Lambda$-$\ell^\infty(\Gamma/\Lambda)$-module structure. Note that $\cK_{g \Lambda} \cong \cH_g$, but again, this isomorphism depends on a choice of section $\Gamma/\Lambda \recht \Gamma$.

\item $\cC_5 \recht \cC_4 : \cK \mapsto \cH$, with $\cH_g = \cK_{g \Lambda}$ and the obvious $\Lambda$-$\Lambda$-$\ell^\infty(\Gamma)$-module structure.
\end{itemlist}

We say that an object $\cH \in \cC_5$ has finite rank if $\cH$ is a finite dimensional Hilbert space. This is equivalent to requiring that all Hilbert spaces $\cH_{g \Lambda}$ are finite dimensional and that there are only finitely many double cosets $\Lambda g \Lambda$ for which $\cH_{g\Lambda}$ is nonzero. Similarly, we say that an object $\cH \in \cC_4$ has finite rank if all Hilbert spaces $\cH_g$ are finite dimensional and if there are only finitely many double cosets $\Lambda g \Lambda$ for which $\cH_g$ is nonzero. Note here that an algebraic variant of the category of finite rank objects in $\cC_4$ was already introduced in \cite{Z98}.

In this way, we have defined the rigid C$^*$-tensor category $\cC_{4,f}(\Lambda < \Gamma)$ consisting of the finite rank objects in $\cC_4$. Note that, in a different context, this rigid C$^*$-tensor category $\cC_{4,f}(\Lambda < \Gamma)$ already appeared in \cite[Section 4]{DV10}.

Denote by $\pi : \Gamma \recht G$ the canonical homomorphism. Identifying $G/K$ and $\Gamma/\Lambda$ and using the homomorphism $\pi : \Lambda \recht K$, every $K$-$L^\infty(G/K)$-module $\cH$ also is a $\Lambda$-$\ell^\infty(\Gamma/\Lambda)$-module. This defines a functor $\cC_2(K < G) \recht \cC_5(\Lambda < \Gamma)$ that is fully faithful because $\pi(\Lambda)$ is dense in $K$. Note however that this fully faithful functor need not be an equivalence of categories: an object $\cH \in \cC_5(\Lambda < \Gamma)$ is isomorphic with an object in the range of this functor if and only if the representation of $\Lambda$ on $\cH$ is of the form $k \mapsto \lambda(\pi(k))$ for a (necessarily unique) continuous representation $\lambda$ of $K$ on $\cH$.

Composing with the equivalence of categories in Proposition \ref{prop.some-equivalences-of-categories}, we have found the fully faithful C$^*$-tensor functor $\Theta : \cC_3(K < G) \recht \cC_4(\Lambda < \Gamma)$, sending finite rank objects to finite rank objects. By construction, $\Theta$ maps the $G$-$L^\infty(G/K)$-$L^\infty(G/K)$-module $L^2(G/K) \ot L^2(G/K)$ (with $G$-action given by $(\lambda_g \ot \lambda_g)_{g \in G}$ and obvious left and right $L^\infty(G/K)$-action) to the $\Lambda$-$\Lambda$-$\ell^\infty(\Gamma)$-module $\ell^2(\Gamma)$.

Next, given the outer action $\Gamma \actson^\al P$, we write $N = P \rtimes \Lambda$ and $M = P \rtimes \Gamma$. Consider the category $\Bimod(N)$ of Hilbert $N$-$N$-bimodules. We define the natural fully faithful C$^*$-tensor functor $\cC_4(\Lambda < \Gamma) \recht \Bimod(N) : \cH \mapsto \cK$ where $\cK = L^2(P) \ot \cH$ and where the $N$-$N$-bimodule structure on $\cK$ is given by
$$(a u_k) \cdot (b \ot \xi) \cdot (d u_{r}) = a \al_k(b)\al_{kh}(d) \ot \lambda(k)\rho(r^{-1}) \xi$$
for all $a,b,d \in P$, $k,r \in \Lambda$, $h \in \Gamma$ and $\xi \in \cH_h$. By construction, this functor maps the $\Lambda$-$\Lambda$-$\ell^\infty(\Gamma)$-module $\ell^2(\Gamma)$ to the $N$-$N$-bimodule $L^2(M)$.

Denoting by $\cC$ the tensor category of finite index $N$-$N$-bimodules generated by the finite index $N$-subbimodules of $L^2(M)$, it follows that $\cC$ is naturally monoidally equivalent to the tensor subcategory $\cC_0$ of $\cC_{3,f}(K < G)$ generated by the finite rank subobjects of $L^2(G/K) \ot L^2(G/K)$. So, it remains to prove that $\cC_0 = \cC_{3,f}(K < G)$. Taking the $n$-th tensor power of $L^2(G/K) \ot L^2(G/K)$ and applying the equivalence between the categories $\cC_{3,f}(K < G)$ and $\cC_{2,f}(K < G)$, it suffices to show that every irreducible $K$-$L^\infty(G/K)$-module appears in one of the $K$-$L^\infty(G/K)$-modules $L^2(G/K) \ot \cdots \ot L^2(G/K)$ with diagonal $G$-action and action of $L^\infty(G/K)$ on the last tensor factor. Reducing with the projections $1_{gK}$, this amounts to proving that for every $g \in G$, every irreducible representation of the compact group $K \cap g K g^{-1}$ appears in a tensor power of $L^2(G/K)$. Because $K < G$ is a Schlichting completion, we have that $\bigcap_{h \in G} hKh^{-1} = \{e\}$ so that the desired conclusion follows.
\end{proof}

\section{\boldmath The tube algebra of $\cC(K < G)$}\label{sec.tube-algebra}

Recall from \cite{Oc93} the following construction of the \emph{tube $*$-algebra} of a rigid C$^*$-tensor category $\cC$ (see also \cite[Section 3]{GJ15} where the terminology \emph{annular algebra} is used, and see as well \cite[Section 3.3]{PSV15}). Whenever $I$ is a full\footnote{Fullness means that every irreducible $i \in \Irr(\cC)$ appears as a subobject of one of the $j \in I$.} family of objects in $\cC$, one defines as follows the $*$-algebra $\cA$ with underlying vector space
$$\cA = \bigoplus_{i,j \in I, \al \in \Irr(\cC)} (i \al ,\al j) \; .$$
Here and in what follows, we denote the tensor product in $\cC$ by concatenation and we denote by $(\be,\gamma)$ the space of morphisms from $\gamma$ to $\be$. By definition, all $(\be,\gamma)$ are finite dimensional Banach spaces. Using the categorical traces $\Tr_\be$ and $\Tr_\gamma$ on $(\be,\be)$, resp.\ $(\gamma,\gamma)$, we turn $(\be,\gamma)$ into a Hilbert space with scalar product
$$\langle V,W \rangle = \Tr_\be(VW^*) = \Tr_\gamma(W^* V) \; .$$
For every $\be \in \cC$, the categorical trace $\Tr_\be$ is defined by using a standard solution for the conjugate equations for $\be$, i.e.\ morphisms $s_\be \in (\be \bebar,\eps)$ and $t_\be \in (\bebar \be,\eps)$ satisfying
$$(s_\be^* \ot 1)(1 \ot t_\be) = 1 \;\; , \;\; (1 \ot s_\be^*)(t_\be \ot 1) = 1 \;\; , \;\; t_\be^* (1 \ot V) t_\be = s_\be^* (V \ot 1) s_\be$$
for all $V \in (\be,\be)$. Then, $\Tr_\be(V) = t_\be^*(1 \ot V)t_\be = s_\be^* (V \ot 1) s_\be$ and $d(\be) = \Tr_\be(1)$ is the categorical dimension of $\be$.

We will also make use of the partial traces
$$\Tr_\be \ot \id : (\be \al,\be \gamma) \recht (\al ,\gamma) : (\Tr_\be \ot \id)(V) = (t_\be^* \ot 1) (1 \ot V) (t_\be \ot 1) \; .$$

Whenever $\cK$ is a Hilbert space, we denote by $\onb(\cK)$ any choice of orthonormal basis in $\cK$. The product in $\cA$ is then defined as follows: for $V \in (i \al, \al j)$ and $W \in (j' \be,\be,k)$, the product $V \cdot W$ equals $0$ when $j \neq j'$ and when $j=j'$, it is equal to
$$V \cdot W = \sum_{\gamma \in \Irr(\cC)} \sum_{U \in \onb(\al \be,\gamma)} d(\gamma) \, (1 \ot U^*)(V \ot 1) (1 \ot W)(U \ot 1) \; .$$
The $*$-operation on $\cA$ is denoted by $V \mapsto V^\#$ and defined by
$$V^\# = (t_\al^* \ot 1)(1 \ot V^* \ot 1)(1 \ot s_\al)$$
for all $V \in (i \al , \al j)$.

The $*$-algebra $\cA$ has a natural positive faithful trace $\Tr$ and for $V \in (i \al, \al j)$, we have that $\Tr(V) = 0$ when $i \neq j$ or $\al \neq \eps$, while $\Tr(V) = \Tr_i(V)$ when $i=j$ and $\al = \eps$, so that $V \in (i,i)$.

Up to strong Morita equivalence, the tube $*$-algebra $\cA$ does not depend on the choice of the full family $I$ of objects in $\cC$, see \cite[Theorem 3.2]{NY15b} and \cite[Section 7.2]{PSV15}. Also note that for an arbitrary object $\al \in \cC$ and $i,j \in I$, we can associate with $V \in (i \al,\al j)$ the element in $\cA$ given by
$$\sum_{\gamma \in \Irr(\cC)} \sum_{U \in \onb(\al,\gamma)} d(\gamma) \, (1 \ot U^*) V (U \ot 1) \; .$$
Although this map $(i \al , \al j) \recht \cA$ is not injective, we will view an element in $V \in (i \al , \al j)$ as an element of $\cA$ in this way.

Formally allowing for infinite direct sums in $\cC$, one defines the C$^*$-tensor category of ind-objects in $\cC$. Later in this section, we will only consider the rigid C$^*$-tensor category $\cC$ of finite rank $G$-$L^\infty(G/K)$-$L^\infty(G/K)$-modules for a given totally disconnected group $G$ with compact open subgroup $K < G$. In that case, the ind-category precisely\footnote{Using Proposition \ref{prop.some-equivalences-of-categories}, every $G$-$L^\infty(G/K)$-$L^\infty(G/K)$-module is a direct sum of finite rank modules because every $K$-$L^\infty(G/K)$-module is a direct sum of finite dimensional modules, which follows because every unitary representation of a compact group is a direct sum of finite dimensional representations.} is the C$^*$-tensor category of all $G$-$L^\infty(G/K)$-$L^\infty(G/K)$-modules. Whenever $\cK_1,\cK_2$ are ind-objects, we denote by $(\cK_1,\cK_2)$ the vector space of \emph{finitely supported} morphisms, where a morphism $V : \cK_2 \recht \cK_1$ is said to be finitely supported if there exist projections $p_i$ of $\cK_i$ onto a finite dimensional subobject (i.e.\ an object in $\cC$) such that $V = p_1 V = V p_2$.

We say that an ind-object $\cH_0$ in $\cC$ is full if every irreducible object $i \in \Irr(\cC)$ is isomorphic with a subobject of $\cH_0$. We define the tube $*$-algebra of $\cC$ with respect to a full ind-object $\cH_0$ as the vector space
$$\cA = \bigoplus_{\al \in \Irr(\cC)} (\cH_0 \al , \al \cH_0)$$
on which the $*$-algebra structure is defined in the same way as above. Note that $(\cH_0,\cH_0)$ naturally is a $*$-subalgebra of $\cA$, given by taking $\al = \eps$ in the above description of $\cA$. In particular, every projection of $p$ of $\cH_0$ on a finite dimensional subobject of $\cH_0$ can be viewed as a projection $p \in \cA$. These projections serve as local units: for every finite subset $\cF \subset \cA$, there exists such a projection $p$ satisfying $p \cdot V = V \cdot p$ for all $V \in \cF$.

Whenever $p_\eps$ is the projection of $\cH_0$ onto a copy of the trivial object $\eps$, we identify $p_\eps \cdot \cA \cdot p_\eps$ with the fusion $*$-algebra $\C[\cC]$ of $\cC$, i.e.\ the $*$-algebra with vector space basis $\Irr(\cC)$, product given by the fusion rules and $*$-operation given by the adjoint object.

To every full family $I$ of objects in $\cC$, we can associate the full ind-object $\cH_0$ by taking the direct sum of all $i \in I$. The tube $*$-algebra of $\cC$ associated with $I$ is then naturally a $*$-subalgebra of the tube $*$-algebra of $\cC$ associated with $\cH_0$. If every irreducible object of $\cC$ appears with finite multiplicity in $\cH_0$, then this inclusion is an equality and both tube $*$-algebras are naturally isomorphic.

For the rest of this section, we fix a totally disconnected group $G$ and a compact open subgroup $K < G$. We denote by $\cC$ the rigid C$^*$-tensor category of all finite rank $G$-$L^\infty(G/K)$-$L^\infty(G/K)$-modules, which we denoted as $\cC_{3,f}(K < G)$ in Section \ref{sec.our-categories}. We determine the tube $*$-algebra $\cA$ of $\cC$ with respect to the following full ind-object.
\begin{equation}\label{eq.favorite-ind-object}
\begin{split}
& \cH_0 = L^2(G \times G/K) \quad\text{with}\\
& (F \cdot \xi) (g,hK) = F(gK) \, \xi(g,hK) \;\; , \;\; (\xi \cdot F)(g,hK) = \xi(g,hK) \, F(ghK) \;\;\text{and}\\ &(\pi(x)\xi)(g,hK) = \xi(x^{-1}g,hK)
\end{split}
\end{equation}
for all $\xi \in L^2(G \times G/K)$, $F \in L^\infty(G/K)$, $x,g \in G$, $hK \in G/K$. Note that every irreducible object of $\cC$ appears with finite multiplicity in $\cH_0$.

We denote by $(\Ad g)_{g \in G}$ the action of $G$ on $G$ by conjugation: $(\Ad g)(h) = g h g^{-1}$. In the rest of this paper, we will make use of the associated full and reduced C$^*$-algebras
$$C_0(G) \rtimes_{\Ad}^f G \quad\text{and}\quad C_0(G) \rtimes_{\Ad}^r G \;\; ,$$
as well as the von Neumann algebra $L^\infty(G) \rtimes_{\Ad} G$. We fix the left Haar measure $\lambda$ on $G$ such that $\lambda(K) = 1$. We equip $L^\infty(G) \rtimes_{\Ad} G$ with the canonical normal semifinite faithful trace $\Tr$ given by
\begin{equation}\label{eq.trace-crossed-product}
\Tr(F \lambda_f) = f(e) \, \int_G F(g) \Delta(g)^{-1/2} \, dg \; .
\end{equation}
Note that the modular function $\Delta$ is affiliated with the center of $L^\infty(G) \rtimes_{\Ad} G$, so that $L^\infty(G) \rtimes_{\Ad} G$ need not be a factor. Also note that the measure used in \eqref{eq.trace-crossed-product} is half way between the left and the right Haar measure of $G$.

We consider the dense $*$-algebra $\Pol(L^\infty(G) \rtimes_{\Ad} G)$ defined as
\begin{equation}\label{eq.pol-algebra}
\begin{split}
\Pol(L^\infty(G) \rtimes_{\Ad} G) = \lspan \{ 1_\cU \, u_x \, p_L \mid \;\; & \cU \subset G \;\;\text{compact open subset}\; , \; x \in G \; , \\ & L < G \;\;\text{compact open subgroup}\}
\end{split}
\end{equation}
and where $p_L \in L(G)$ denotes the projection onto the $L$-invariant vectors, i.e.
$$p_L = \lambda(L)^{-1} \int_L u_k \, dk \; .$$
Note that $\Pol(L^\infty(G) \rtimes_{\Ad} G)$ equals the linear span of all $F \lambda_f$ where $F$ and $f$ are continuous, compactly supported, locally constant functions on $G$.

We now identify the tube $*$-algebra of $\cC$ with $\Pol(L^\infty(G) \rtimes_{\Ad} G)$. For every $x \in G$ and every irreducible representation $\pi : K \cap x K x^{-1} \recht \cU(\cK)$, we denote by $\cH(\pi,x) \in \Irr(\cC)$ the irreducible $G$-$L^\infty(G/K)$-$L^\infty(G/K)$-module such that $\pi$ is isomorphic with the representation of $K \cap x K x^{-1}$ on $1_{xK} \cdot \cH(\pi,x) \cdot 1_{e K}$. Note that this gives us the identification
\begin{equation}\label{eq.identify-irr-C}
\Irr(\cC) = \{(\pi,x) \mid x \in K\backslash G / K \; , \; \pi \in \Irr(K \cap x K x^{-1}) \} \; .
\end{equation}
We denote by $\chi_\pi$ the character of $\pi$, i.e.\ the locally constant function with support $K \cap x K x^{-1}$ and $\chi_\pi(k) = \Tr(\pi(k))$ for all $k \in K \cap x K x^{-1}$.

\begin{theorem}\label{thm.our-tube-alg}
The $G$-$L^\infty(G/K)$-$L^\infty(G/K)$-module $\cH_0$ introduced in \eqref{eq.favorite-ind-object} is full. There is a natural $*$-anti-isomorphism $\Theta$ of the associated tube $*$-algebra $\cA$ onto $\Pol(L^\infty(G) \rtimes_{\Ad} G)$. The $*$-anti-isomorphism $\Theta$ is trace preserving.

Denoting by $p_\eps$ the projection in $\cA$ that corresponds to the unique copy of the trivial object $\eps$ in $\cH_0$ and identifying $p_\eps \cdot \cA \cdot p_\eps$ with the fusion $*$-algebra of $\cC$, we have that $\Theta(p_\eps) = 1_K p_K$ and that the restriction of $\Theta$ to $\C[\cC]$ is given by
\begin{equation}\label{eq.Theta-on-fusion-algebra}
d(\pi,x)^{-1} \Theta(\pi,x) = p_K \, \dim(\pi)^{-1} \chi_\pi \, u_x \, p_K \; ,
\end{equation}
where $d(\pi,x)$ denotes the categorical dimension of $(\pi,x) \in \Irr(\cC)$ and $\dim(\pi)$ denotes the ordinary dimension of the representation $\pi$.
\end{theorem}

\begin{proof}
To see that $\cH_0$ is full, it suffices to observe that for every $h \in G$, the unitary representation of $K \cap h Kh^{-1}$ on $1_{eK} \cdot \cH_0 \cdot 1_{hK}$ contains the regular representation of $K \cap h Kh^{-1}$.

Assume that $\Psi : C_0(G) \rtimes^f_{\Ad} G \recht B(\cK)$ is any nondegenerate $*$-representation. As follows, we associate with $\Psi$ a unitary half braiding\footnote{Formally, a unitary half braiding is an object in the Drinfeld center of ind-$\cC$. More concretely, a unitary half braiding consists of an underlying ind-object $\cK_1$ together with natural unitary isomorphisms $\cH  \cK_1 \recht \cK_1  \cH$ for all objects $\cH$. We refer to \cite[Section 2.1]{NY15a} for further details.} on ind-$\cC$. Whenever $\cH$ is a $G$-$L^\infty(G/K)$-$L^\infty(G/K)$-module, we consider a new $G$-$L^\infty(G/K)$-$L^\infty(G/K)$-module with underlying Hilbert space $\cK \ot \cH$ and structure maps
$$\pi_{\cK \ot \cH}(g) = \Psi(g) \ot \pi_\cH(g) \;\; , \;\; \lambda_{\cK \ot \cH}(F) = (\Psi \ot \lambda_\cH)\Delta(F) \;\; , \;\; \rho_{\cK \ot \cH}(F) = 1 \ot \rho_\cH(F) \;\; ,$$
for all $g \in G$, $F \in L^\infty(G/K)$, with $\Delta(F)(g,hK) = F(ghK)$.

We similarly turn $\cH \ot \cK$ into a $G$-$L^\infty(G/K)$-$L^\infty(G/K)$-module with structure maps
$$\pi_{\cH \ot \cK}(g) = \pi_\cH(g) \ot \Psi(g) \;\; , \;\; \lambda_{\cH \ot \cK}(F) = \lambda_\cH(F) \ot 1 \;\; , \;\; \rho_{\cH \ot \cK}(F) = (\rho_\cH \ot \Psi)\Deltatil(F) \;\;,$$
where $\Deltatil(F)(gK,h) = F(h^{-1} gK)$.

Defining the unitary $U \in M(C_0(G) \ot K(\cH))$ given by $U(x) = \pi_\cH(x)$ for all $x \in G$ and denoting by $\Sigma : \cK \ot \cH \recht \cH \ot \cK$ the flip map, one checks that $\Sigma (\Psi \ot \id)(U)$ is an isomorphism between the $G$-$L^\infty(G/K)$-$L^\infty(G/K)$-modules $\cK \ot \cH$ and $\cH \ot \cK$. So, defining
$$\cK_1 := \cK \ot L^2(G/K) \cong L^2(G/K) \ot \cK \; ,$$
we have found the $G$-$L^\infty(G/K)$-$L^\infty(G/K)$-module $\cK_1$ with the property that for every $G$-$L^\infty(G/K)$-$L^\infty(G/K)$-module $\cH$, there is a natural unitary isomorphism
$$\sigma_\cH : \cH \cK_1 \recht \cK_1 \cH \; .$$
Here and in what follows, we denote by concatenation the tensor product in the category of $G$-$L^\infty(G/K)$-$L^\infty(G/K)$-modules. So, $\sigma$ is a unitary half braiding for ind-$\cC$.

Using the ind-object $\cH_0$ defined in \eqref{eq.favorite-ind-object} and recalling that $\cK_1 \overline{\cH_0} = \cK \ot \overline{\cH_0}$ as Hilbert spaces, we define the Hilbert space
$$\cK_2 = (\cK \ot \overline{\cH_0},\eps)$$
and we consider the tube $*$-algebra $\cA$ associated with $\cH_0$. Using standard solutions for the conjugate equations, there is a natural linear bijection
$$V \in (\cH_0  \cH , \cH \cH_0) \mapsto \Vtil \in (\cH  \overline{\cH_0} , \overline{\cH_0} \cH)$$
between finitely supported morphisms.

By \cite[Proposition 3.14]{PSV15} and using the partial categorical trace $\Tr_\cH \ot \id \ot \id$, the unitary half braiding $\sigma$ gives rise to a nondegenerate $*$-anti-homomorphism $\Theta : \cA \recht B(\cK_2)$ given by
\begin{equation}\label{eq.first-formula-Theta}
\Theta(V)\xi = (\Tr_\cH \ot \id \ot \id)\bigl( (\sigma_\cH^* \ot 1)(1 \ot \Vtil)(\xi \ot 1)\bigr)
\end{equation}
for all $\cH \in \cC$, $\xi \in \cK_2$ and all finitely supported $V \in (\cH_0 \cH , \cH \cH_0)$.

We now compute the expression in \eqref{eq.first-formula-Theta} more concretely. Whenever $h \in G$ and $K_0 < K$ is an open subgroup such that $hK_0 h^{-1} \subset K$, we define the finite rank $G$-$L^\infty(G/K)$-$L^\infty(G/K)$-module $L^2(G/K_0)_h$ with underlying Hilbert space $L^2(G/K_0)$ and structure maps
$$(x \cdot \xi)(gK_0) = \xi(x^{-1}gK_0) \;\; , \;\; (F_1 \cdot \xi \cdot F_2)(g K_0) = F_1(g K) \, \xi(g K_0) \, F_2(g h^{-1} K) \; .$$
Note that there is a natural isomorphism $\overline{L^2(G/K_0)_h} \cong L^2(G/K_0)_{h^{-1}}$. Letting $K_0$ tend to $\{e\}$, the direct limit of $L^2(G/K_0)_{h^{-1}}$ becomes $L^2(G)_{h^{-1}}$. Since $\cH_0 = \bigoplus_{h \in G/K} L^2(G)_{h^{-1}}$, we identify
$$\overline{\cH_0} = \bigoplus_{h \in G/K} L^2(G)_h$$
and we view $L^2(G/K_0)_h \subset \overline{\cH_0}$ whenever $h \in G$ and $K_0 < K \cap h^{-1} K h$ is an open subgroup.

The Hilbert space $\cK_2$ equals the space of $K$-invariant vectors in $1_{eK} \cdot (\cK \ot \overline{\cH_0}) \cdot 1_{eK}$. In this way, the space of $K$-invariant vectors in $1_{eK} \cdot (\cK \ot L^2(G/K_0)_h) \cdot 1_{eK}$ naturally is a subspace of $\cK_2$. But this last space of $K$-invariant vectors can be unitarily identified with $\Psi(1_{Kh^{-1}} \, p_{h K_0 h^{-1}}) \cK$ by sending the vector $\xi_0 \in \Psi(1_{Kh^{-1}} \, p_{h K_0 h^{-1}}) \cK$ to the vector
$$\Delta(h)^{-1/2} \sum_{k \in K/hK_0h^{-1}} \Psi(k) \xi_0 \ot 1_{khK_0} \in \cK \ot L^2(G/K_0) \; .$$
We now use that for every $\cH \in \cC$, the categorical trace $\Tr_\cH$ on $(\cH,\cH)$ is given by
\begin{align*}
\Tr_\cH(V) & =  \sum_{x \in G/K , \eta \in \onb(1_{xK} \cdot \cH \cdot 1_{eK})} \Delta(x)^{-1/2} \, \langle V \eta,\eta \rangle \\ & =  \sum_{y \in G/K , \eta \in \onb(1_{eK} \cdot \cH \cdot 1_{yK})} \Delta(y)^{-1/2} \, \langle V \eta,\eta \rangle \; .
\end{align*}
A straightforward computation then gives that for all $\cH \in \cC$ and all
$$V \in \bigl(\;\overline{L^2(G/K_0)_g} \; \cH \; , \; \cH \; \overline{L^2(G/K_1)_h}\;\bigr)$$
with $g,h \in G$ and $K_0 < K \cap g^{-1} K g$, $K_1 < K \cap h^{-1} K h$ open subgroups, we have
\begin{equation}\label{eq.concrete-formula}
\begin{split}
& \Theta(V) = \Delta(g)^{-1/2} \, \Delta(h)^{1/2} \, [K : K_1] \\ & \sum_{\substack{x \in G/gK_0g^{-1} \\  y \in K /K_2 \\ \eta \in \onb(1_{xK} \cdot \cH \cdot 1_{eK})}}  \Delta(x)^{-1/2} \; \Psi(1_{K_2 y^{-1} h^{-1}} \; u_x \; p_{g K_0 g^{-1}}) \; \langle \Vtil (1_{xg K_0} \ot \eta) , \pi_\cH(h y) \eta \ot 1_{h K_1} \rangle \; ,
\end{split}
\end{equation}
whenever $K_2 < K$ is a small enough open subgroup such that $\pi_\cH(k)$ is the identity on $\cH \cdot 1_{eK}$ for all $k \in K_2$. Note that because $\cH$ has finite rank, such an open subgroup $K_2$ exists. Also, there are only finitely many $x \in G/K$ such that $1_{xK} \cdot \cH \cdot 1_{eK}$ is nonzero. Therefore, the sum appearing in \eqref{eq.concrete-formula} is finite.

Applying this to the regular representation $C_0(G) \rtimes^f_{\Ad} G \recht B(L^2(G \times G))$, we see that \eqref{eq.concrete-formula} provides a $*$-anti-homomorphisms $\Theta$ from $\cA$ to $\Pol(L^\infty(G) \rtimes_{\Ad} G)$. A direct computation gives that $\Theta$ is trace preserving, using the trace $\Tr$ on $L^\infty(G) \rtimes_{\Ad} G$ defined in \eqref{eq.trace-crossed-product}. In particular, $\Theta$ is injective.

We now prove that $\Theta$ is surjective. Fix elements $g,h,\al \in G$ satisfying $\al g = h \al$. Choose any open subgroup $K_0 < K$ such that $g K_0 g^{-1}$, $\al K_0 \al^{-1}$ and $K_1 := h^{-1} \al K_0 \al^{-1} h$ are all subgroups of $K$. Put $\cH = L^2(G/K_0)_\al$ and note that $\cH$, $L^2(G/K_0)_g$ and $L^2(G/K_1)_h$ are well defined objects in $\cC$. For every $k \in K$, we consider the vectors
\begin{align*}
& 1_{k \al g K_0} \ot 1_{k \al K_0} \in 1_{k \al g K} \cdot \bigl( L^2(G/K_0)_g \; \cH \bigr) \cdot 1_{eK} \quad\text{and} \\
& 1_{k h \al K_0} \ot 1_{k h K_1} \in 1_{k \al g K} \cdot \bigl( \cH \; L^2(G/K_1)_h \bigr) \cdot 1_{eK} \; .
\end{align*}
In both cases, we get an orthogonal family of vectors indexed by
$$k \in K / (K \cap \al K_0 \al^{-1} \cap \al g K_0 (\al g)^{-1}) \; .$$
So, we can uniquely define $V \in \bigl(\;\overline{L^2(G/K_0)_g} \; \cH \; , \; \cH \; \overline{L^2(G/K_1)_h}\;\bigr)$ such that the restriction of $\Vtil$ to $\bigl( L^2(G/K_0)_g \; \cH \bigr) \cdot 1_{eK}$ is the partial isometry given by
$$1_{k \al g K_0} \ot 1_{k \al K_0} \mapsto \Delta(\al)^{-1/2} \, \Delta(h)^{-1/2} \, 1_{k h \al K_0} \ot 1_{k h K_1} \quad\text{for all}\;\; k \in K \; .$$
A direct computation gives that $\Theta(V)$ is equal to a nonzero multiple of
\begin{equation}\label{eq.we-have-these}
1_{\al K_0 \al^{-1} h^{-1}} \; u_\al \; p_{g K_0 g^{-1}} \; .
\end{equation}
From \eqref{eq.concrete-formula}, we also get that $\Theta$ maps $(\cH_0,\cH_0) \subset \cA$ onto $\Pol(L^\infty(K \backslash G) \rtimes K)$, defined as the linear span of all
$$1_{K x} \; u_k \; p_L$$
with $x \in G$, $k \in K$ and $L < K$ an open subgroup. In combination with \eqref{eq.we-have-these}, it follows that $\Theta$ is surjective.

Finally, by restricting \eqref{eq.concrete-formula} to the cases where $g=h=e$ and $K_0=K_1=K$, we find that \eqref{eq.Theta-on-fusion-algebra} holds.
\end{proof}

We recall from \cite{PV14} the notion of a \emph{completely positive (cp) multiplier} on a rigid C$^*$-tensor category $\cC$. By \cite[Proposition 3.6]{PV14}, to every function $\vphi : \Irr(\cC) \recht \C$ is associated a system of linear maps
\begin{equation}\label{eq.system-of-linear-maps}
\Psi^\vphi_{\al_1 | \beta_1,\al_2 | \beta_2} : (\al_1 \be_1, \al_2 \be_2) \recht (\al_1 \be_1, \al_2 \be_2) \quad\text{for all}\;\; \al_i,\be_i \in \cC
\end{equation}
satisfying
$$\Psi^\vphi_{\al_3 | \beta_3,\al_4 | \beta_4} ((X \ot Y) V (Z \ot T)) = (X \ot Y) \; \Psi^\vphi_{\al_1 | \beta_1,\al_2 | \beta_2}(V) \; (Z \ot T)$$
for all $X \in (\al_3,\al_1)$, $Y \in (\be_3,\be_1)$, $Z \in (\al_2,\al_4)$, $T \in (\be_2,\be_4)$, as well as
\begin{align*}
& \Psi^\vphi_{\al | \albar,\eps | \eps}(s_\al) = \vphi(\al) \, s_\al \quad\text{and} \\
& \Psi^\vphi_{\al_1 \al_2 | \be_2 \be_1, \al_3 \al_4 | \be_4 \be_3}(1 \ot V \ot 1) = 1 \ot \Psi^\vphi_{\al_2 | \be_2, \al_4 | \be_4}(V) \ot 1
\end{align*}
for all $V \in (\al_2 \be_2, \al_4 \be_4)$.

\begin{definition}[{\cite[Definition 3.4]{PV14}}]\label{def.cp-cb-multiplier}
Let $\cC$ be a rigid C$^*$-tensor category.
\begin{itemlist}
\item A \emph{cp-multiplier} on $\cC$ is a function $\vphi : \Irr(\cC) \recht \C$ such that the maps $\Psi^\vphi_{\al | \be,\al | \be}$ on $(\al\be,\al\be)$ are completely positive for all $\al,\be \in \cC$.

\item A cp-multiplier $\vphi : \Irr(\cC) \recht \C$ is said to be $c_0$ if the function $\vphi : \Irr(\cC) \recht \C$ tends to zero at infinity.

\item A \emph{cb-multiplier} on $\cC$ is a function $\vphi : \Irr(\cC) \recht \C$ such that
$$\|\vphi\|\cb := \sup_{\al_i,\be_i \in \cC} \|\Psi^\vphi_{\al_1 | \be_1,\al_2 | \be_2}\|\cb < \infty \; .$$
\end{itemlist}
\end{definition}

A function $\vphi : \Irr(\cC) \recht \C$ gives rise to the following linear functional $\om_\vphi : \cA \recht \C$ on the tube algebra $\cA$ of $\cC$ with respect to any full family of objects containing once the trivial object $\eps$:
$$\om_\vphi : \cA \recht \C : \om_\vphi(V) = \begin{cases} d(\al) \, \vphi(\al) &\quad\text{if $V = 1_\al \in (\eps \al , \al \eps)$,} \\
0 &\quad\text{if $V \in (i \al, \al j)$ with $i \neq \eps$ or $j \neq \eps$.}\end{cases}$$
By \cite[Theorem 6.6]{GJ15}, the function $\vphi : \Irr(\cC) \recht \C$ is a cp-multiplier in the sense of Definition \ref{def.cp-cb-multiplier} if and only if $\om_\vphi$ is positive on $\cA$ in the sense that $\om_\vphi(V \cdot V^\#) \geq 0$ for all $V \in \cA$. In Proposition \ref{prop.equality-of-cb-norms}, we prove a characterization of cb-multipliers in terms of completely bounded multipliers of the tube $*$-algebra.

From Theorem \ref{thm.our-tube-alg}, we then get the following result. We again denote by $\cC$ be the rigid C$^*$-tensor category of finite rank $G$-$L^\infty(G/K)$-$L^\infty(G/K)$-modules and we identify $\Irr(\cC)$ as in \eqref{eq.identify-irr-C} with the set of pairs $(\pi,x)$ where $x \in K \backslash G / K$ and $\pi$ is an irreducible representation of the compact group $K \cap x K x^{-1}$. In order to identify the $c_0$ cp-multipliers on $\cC$, we introduce the following definition.

\begin{definition}\label{def.c-0-state}
We say that a complex measure $\mu$ on $G$ (i.e.\ an element of $C_0(G)^*$) is $c_0$ if
$$\lambda(\mu) := \int_G \lambda_g \; d\mu(g) \in L(G)$$
belongs to $C^*_r(G)$.

We say that a positive functional $\om$ on $C_0(G) \rtimes^f_{\Ad} G$ is $c_0$ if for every $x \in G$, the complex measure $\mu_x$ defined by $\mu_x(F) = \om(F u_x)$ for all $F \in C_0(G)$ is $c_0$ and if the function $G \recht C^*_r(G) : x \mapsto \lambda(\mu_x)$ tends to zero at infinity, i.e.\ $\lim_{x \recht \infty} \|\lambda(\mu_x)\| = 0$.
\end{definition}

\begin{proposition}\label{prop.cp-multipliers-vs-states}
The formula
\begin{equation}\label{eq.formula-cp-mult-vs-states}
\vphi(\pi,x) = \om(p_K \, \dim(\pi)^{-1} \chi_\pi \, u_x \, p_K)
\end{equation}
gives a bijection between the cp-multipliers $\vphi$ on $\Irr(\cC)$ and the positive functionals $\om$ on the C$^*$-algebra $q (C_0(G) \rtimes^f_{\Ad} G) q$, where $q = 1_K p_K$.

The cp-multiplier $\vphi$ is $c_0$ if and only if the positive functional $\om$ is $c_0$ in the sense of Definition \ref{def.c-0-state}.

Using the notations $C_u(\cC)$ and $C_r(\cC)$ of \cite[Definition 4.1]{PV14} for the universal and reduced C$^*$-algebra of $\cC$, we have the natural anti-isomorphisms $C_u(\cC) \recht q (C_0(G) \rtimes^f_{\Ad} G) q$ and $C_r(\cC) \recht q (C_0(G) \rtimes^r_{\Ad} G) q$.
\end{proposition}

\begin{proof}
Note that the $G$-$L^\infty(G/K)$-$L^\infty(G/K)$-module $\cH_0$ in \eqref{eq.favorite-ind-object} contains exactly once the trivial module. The first part of the proposition is then a direct consequence of Theorem \ref{thm.our-tube-alg} and the above mentioned characterization \cite{GJ15} of cp-multipliers as positive functionals on the tube $*$-algebra. The isomorphisms for $C_u(\cC)$ and $C_r(\cC)$ follow in the same way.

Fix a positive functional $\om$ on $q (C_0(G) \rtimes^f_{\Ad} G) q$ with corresponding cp-multiplier $\vphi : \Irr(\cC) \recht \C$ given by \eqref{eq.formula-cp-mult-vs-states}. We extend $\om$ to $C_0(G) \rtimes^f_{\Ad} G$ by $\om(T) = \om(qTq)$. For every $x \in G$, define $\mu_x \in C_0(G)^*$ given by $\mu_x(F) = \om(F u_x)$ for all $F \in C_0(G)$. Note that $\mu_x$ is supported on $K \cap x K x^{-1}$ and that $\mu_x$ is $\Ad(K \cap x K x^{-1})$-invariant. Therefore, $\lambda(\mu_x) \in \cZ(L(K \cap x K x^{-1}))$. For every $\pi \in \Irr(K \cap x K x^{-1})$, denote by $z_\pi \in \cZ(L(K \cap x K x^{-1}))$ the corresponding minimal central projection. From \eqref{eq.formula-cp-mult-vs-states}, we get that
\begin{equation}\label{eq.tussenformule}
\lambda(\mu_x) z_\pi = \vphi(\pi,x) z_\pi \quad\text{for all}\;\; x \in G \; , \; \pi \in \Irr(K \cap x K x^{-1}) \; .
\end{equation}
For a fixed $x \in G$, an element $T \in \cZ(L(K \cap x K x^{-1}))$ belongs to $C^*_r(G)$ if and only if $T \in C^*_r(K \cap x K x^{-1})$ if and only if $\lim_{\pi \recht \infty} \|T z_\pi\| = 0$. Also, $\|T\| = \sup_{\pi \in \Irr(K \cap x K x^{-1})} \|T z_\pi\|$. So by \eqref{eq.tussenformule}, we get that $\mu_x$ is $c_0$ if and only if
\begin{equation}\label{eq.fixed-x}
\lim_{\pi \recht \infty} |\vphi(\pi,x)| = 0
\end{equation}
and that $\om$ is a $c_0$ functional if and only if \eqref{eq.fixed-x} holds for all $x \in G$ and we moreover have that
$$\lim_{x \recht \infty} \bigl( \sup_{\pi \in \Irr(K \cap x K x^{-1})} |\vphi(\pi,x)| \bigr) = 0 \; .$$
Altogether, it follows that $\om$ is a $c_0$ functional in the sense of Definition \ref{def.c-0-state} if and only if $\vphi$ is a $c_0$-function.
\end{proof}

For later use, we record the following lemma.

\begin{lemma}\label{lem.ac-c0}
Let $\mu$ be a probability measure on $G$ that is $c_0$ in the sense of Definition \ref{def.c-0-state}. Then every complex measure $\om \in C_0(G)^*$ that is absolutely continuous with respect to $\mu$ is still $c_0$.
\end{lemma}
\begin{proof}
Denote by $C_c(G)$ the space of continuous compactly supported functions on $G$. Since $C_c(G) \subset L^1(G,\mu)$ is dense, it is sufficient to prove that $F \cdot \mu$ is $c_0$ for every $F \in C_c(G)$. Denote by $\om_F \in C^*_r(G)^*$ the functional determined by $\om_F(\lambda_x) = F(x)$ for all $x \in G$. Denote by $\Deltah : C^*_r(G) \recht M(C^*_r(G) \ot C^*_r(G))$ the comultiplication determined by $\Deltah(\lambda_x) = \lambda_x \ot \lambda_x$. Recall that for every $\om \in C^*_r(G)^*$ and every $T \in C^*_r(G)$, we have that $(\om \ot \id)\Deltah(T) \in C^*_r(G)$. Since
$$\lambda(F \cdot \mu) = (\om_F \ot \id)\Deltah(\lambda(\mu)) \; ,$$
the lemma is proven.
\end{proof}

\section{\boldmath Haagerup property and property (T) for $\cC(K < G)$}\label{sec.haagerup-T}

In Definition \ref{def.cp-cb-multiplier}, we already recalled the notion of a cp-multiplier $\vphi : \Irr(\cC) \recht \C$ on a rigid C$^*$-tensor category $\cC$. In terms of cp-multipliers, \emph{amenability} of a rigid C$^*$-tensor category, as defined in \cite{Po94,LR96}, amounts to the existence of finitely supported cp-multipliers $\vphi_n : \Irr(\cC) \recht \C$ that converge to $1$ pointwise, see \cite[Proposition 5.3]{PV14}. Following \cite[Definition 5.1]{PV14}, a rigid C$^*$-tensor category $\cC$ has the \emph{Haagerup property} if there exist $c_0$ cp-multipliers $\vphi_n : \Irr(\cC) \recht \C$ that converge to $1$ pointwise, while $\cC$ has \emph{property (T)} if all cp-multipliers converging to $1$ pointwise, must converge to $1$ uniformly.

Similarly, when $\cC_1$ is a full C$^*$-tensor subcategory of $\cC$, we say that $\cC_1 \subset \cC$ has the \emph{relative property (T)} if all cp-multipliers on $\cC$ converging to $1$ pointwise, must converge to $1$ uniformly on $\Irr(\cC_1) \subset \Irr(\cC)$.

We now turn back to the rigid C$^*$-tensor category $\cC$ of finite rank $G$-$L^\infty(G/K)$-$L^\infty(G/K)$-modules, where $G$ is a totally disconnected group $G$ and $K < G$ is a compact open subgroup. Note that $\Rep K$ is a full C$^*$-tensor subcategory of $\cC$, consisting of the $G$-$L^\infty(G/K)$-$L^\infty(G/K)$-modules $\cH$ with the property that $1_{xK} \cdot \cH \cdot 1_{eK}$ is zero for all $x \not\in K$.

Recall from Definition \ref{def.c-0-state} the notion of a $c_0$ complex measure on $G$. We identify the space of complex measures with $C_0(G)^*$ and we denote by $\cS(C_0(G)) \subset C_0(G)^*$ the state space of $C_0(G)$, i.e.\ the set of probability measures on $G$.

\begin{theorem}\label{thm.amenable-haagerup-T}
Let $G$ be a totally disconnected group and $K < G$ a compact open subgroup. Denote by $\cC$ the rigid C$^*$-tensor category of finite rank $G$-$L^\infty(G/K)$-$L^\infty(G/K)$-modules.
\begin{enumlist}
\item $\cC$ is amenable if and only if $G$ is amenable.
\item $\cC$ has the Haagerup property if and only if $G$ has the Haagerup property and there exists a sequence of $c_0$ probability measures $\mu_n \in \cS(C_0(G))$ such that $\mu_n \recht \delta_e$ weakly$^*$ and such that $\|\mu_n \circ \Ad x - \mu_n\| \recht 0$ uniformly on compact sets of $x \in G$.
\item $\cC$ has property (T) if and only if $G$ has property (T) and every sequence sequence of $\Ad G$-invariant probability measures $\mu_n \in \cS(C_0(G))$ that converges to $\delta_e$ weakly$^*$ must converge in norm.
\item $\Rep K \subset \cC$ has the relative property (T) if and only if every sequence of probability measures $\mu_n \in \cS(C_0(G))$ such that $\mu_n \recht \delta_e$ weakly$^*$ and $\|\mu_n \circ \Ad x - \mu_n\| \recht 0$ uniformly on compact sets of $x \in G$ satisfies $\|\mu_n - \delta_e\| \recht 0$.
\end{enumlist}
\end{theorem}
\begin{proof}
Denote by $\counit : C_0(G) \rtimes^f_{\Ad} G \recht \C$ the character given by $\counit(F \lambda_f) = F(e) \int_G f(x) \, dx$. Write $q = 1_K p_K$.

1. Combining Proposition \ref{prop.cp-multipliers-vs-states} and \cite[Proposition 5.3]{PV14}, we get that $\cC$ is amenable if and only if the canonical $*$-homomorphism $q (C_0(G) \rtimes^f_{\Ad} G) q \recht q (C_0(G) \rtimes^r_{\Ad} G) q$ is an isomorphism. This holds if and only if $G$ is amenable.

2. First assume that $\cC$ has the Haagerup property. By Proposition \ref{prop.cp-multipliers-vs-states}, we find a sequence of states $\om_n$ on $q (C_0(G) \rtimes^f_{\Ad} G) q$ such that $\om_n \recht \counit$ weakly$^*$ and such that every $\om_n$ is a $c_0$ state in the sense of Definition \ref{def.c-0-state}. For every $x \in G$, define $\mu_n(x) \in C_0(G)^*$ given by $\mu_n(x)(F) = \om_n(F u_x)$.

Using the strictly continuous extension of $\om_n$ to the multiplier algebra $M(C_0(G) \rtimes^f_{\Ad} G)$, we get that $x \mapsto \om_n(u_x)$ is a sequence of continuous positive definite functions converging to $1$ uniformly on compact subsets of $G$. We claim that for every fixed $n$, the function $x \mapsto \om_n(x)$ tends to $0$ at infinity. Denote by $\counit_K : C^*_r(G) \recht \C$ the state given by composing the conditional expectation $C^*_r(G) \recht C^*_r(K)$ with the trivial representation $\counit : C^*_r(K) \recht \C$. Then,
$$\om_n(x) = \counit_K(\lambda(\mu_n(x)))$$
and the claim is proven. So, $G$ has the Haagerup property.

The restriction of $\om_n$ to $C_0(G)$ provides a sequence of $c_0$ probability measures $\mu_n \in \cS(C_0(G))$ such that $\mu_n \recht \delta_e$ weakly$^*$ and $\|\mu_n \circ \Ad x - \mu_n\| \recht 0$ uniformly on compact sets of $x \in G$.

Conversely assume that $G$ has the Haagerup property and that $\mu_n$ is such a sequence of probability measures. By restricting $\mu_n$ to $K$, normalizing and integrating $\int_K (\mu_n \circ \Ad k ) \, dk$, we may assume that the probability measures $\mu_n$ are supported on $K$ and are $\Ad K$-invariant. Fix a strictly positive right $K$-invariant function $w : G \recht \R_0^+$ with $\int_G w(g) \, dg = 1$. Define the probability measures $\mutil_n$ on $G$ given by
$$\mutil_n = \int_G w(g) \, \mu_n \circ \Ad g \; dg \; .$$
Note that $\mutil_n$ is still $\Ad K$-invariant. Also,
$$\lambda(\mutil_n) = \int_G w(g) \, \lambda_g^* \, \lambda(\mu_n) \, \lambda_g \; dg$$
so that each $\mutil_n$ is a $c_0$ probability measure.

By construction, for every $x \in G$, the measure $\mutil_n \circ \Ad x$ is absolutely continuous with respect to $\mutil_n$. We denote by $\Delta_n(x)$ the Radon-Nikodym derivative and define the unitary representations
$$\theta_n : G \recht \cU(L^2(G,\mutil_n)) : \theta_n(x) \xi = \Delta_n(x)^{1/2} \, \xi \circ \Ad x^{-1} \; .$$
We also define $\theta_n : C_0(G) \recht B(L^2(G,\mutil_n))$ given by multiplication operators and we have thus defined a nondegenerate $*$-representation of $C_0(G) \rtimes^f_{\Ad} G$ on $L^2(G,\mutil_n)$.

Note that $\mu_n$ is absolutely continuous with respect to $\mutil_n$. We denote by $\zeta_n \in L^2(G,\mutil_n)$ the square root of the Radon-Nikodym derivative of $\mu_n$ with respect to $\mutil_n$. Since both $\mu_n$ and $\mutil_n$ are $\Ad K$-invariant, we get that $\theta_n(p_K) \zeta_n = \zeta_n$. Since $\mu_n$ is supported on $K$, also $\zeta_n$ is supported on $K$ meaning that $\theta(1_K)\zeta_n = \zeta_n$.

Since $G$ has the Haagerup property, we can also fix a unitary representation $\pi : G \recht \cU(\cH)$ and a sequence of $\pi(K)$-invariant unit vectors $\xi_n \in \cH$ such that $\|\pi(x) \xi_n - \xi_n \| \recht 0$ uniformly on compact sets of $x \in G$ and, for every fixed $n$, the function $x \mapsto \langle \pi(x) \xi_n,\xi_n \rangle$ tends to zero at infinity.

The formulas $\psi(x) = \theta_n(x) \ot \pi(x)$ and $\psi(F) = \theta(F) \ot 1$ define a nondegenerate $*$-representation of $C_0(G) \rtimes^f_{\Ad} G$ on $L^2(G,\mutil_n) \ot \cH$. We define the states $\om_n$ on $C_0(G) \rtimes^f_{\Ad} G$ given by $\om_n(T) = \langle \psi(T) (\zeta_n \ot \xi_n) , \zeta_n \ot \xi_n \rangle$. By construction, $\om_n(q) = 1$ for all $n$ and $\om_n \recht \counit$ weakly$^*$. It remains to prove that each $\om_n$ is a $c_0$ state. Proposition \ref{prop.cp-multipliers-vs-states} then gives that $\cC$ has the Haagerup property.

Fix $n$. Defining $\mu_n(x) \in C_0(G)^*$ given by $\mu_n(x)(F) = \om_n(F u_x)$, we get that
$$\mu_n(x)(F) = \langle \theta_n(F) \, \theta(x) \, \zeta_n , \zeta_n \rangle \; \langle \pi(x) \xi_n, \xi_n \rangle \; .$$
Since the function $x \mapsto \langle \pi(x) \xi_n, \xi_n \rangle$ tends to zero at infinity, we get that even $x \mapsto \|\mu_n(x)\|$ tends to zero at infinity. So, we only have to show that for every fixed $x$, the complex measure given by $F \mapsto \langle \theta_n(F) \, \theta(x) \, \zeta_n , \zeta_n \rangle$ is $c_0$. By construction, this complex measure is absolutely continuous with respect to $\mutil_n$. The conclusion then follows from Lemma \ref{lem.ac-c0}.

3. Note that it follows from \cite[Proposition 5.5]{PV14} that $\cC$ has property (T) if and only if every sequence of states on $q(C_0(G) \rtimes^f_{\Ad} G)q$ converging weakly$^*$ to $\counit$ must converge to $\counit$ in norm.

First assume that $\cC$ has property~(T). Both states on $C^*(G)$ and $\Ad G$-invariant states on $C_0(G)$ give rise to states on $C_0(G) \rtimes^f_{\Ad} G$. One implication of 3 thus follows immediately. Conversely assume that $G$ has property (T) and that every sequence of $\Ad G$-invariant probability measures $\mu_n \in \cS(C_0(G))$ converging weakly$^*$ to $\delta_e$ must converge in norm to $\delta_e$. Let $\om_n$ be a sequence of states on $q(C_0(G) \rtimes^f_{\Ad} G)q$ converging to $\counit$ weakly$^*$. Let $p \in C^*(G)$ be the Kazhdan projection. Replacing $\om_n$ by $\om_n(p)^{-1} \, p \cdot \om_n \cdot p$, we may assume that $\om_n$ is left and right $G$-invariant. This means that $\om_n(F u_x) = \mu_n(F)$ for all $F \in C_0(G)$, $x \in G$, where $\mu_n$ is a sequence of $\Ad G$-invariant probability measures on $G$ converging weakly$^*$ to $\delta_e$. Thus $\|\mu_n - \delta_e\| \recht 0$ so that $\|\om_n - \counit\| \recht 0$.

4. First assume that $\Rep K \subset \cC$ has the relative property (T) and take a sequence of probability measures $\mu_n \in \cS(C_0(G))$ such that $\mu_n \recht \delta_e$ weakly$^*$ and $\|\mu_n \circ \Ad x - \mu_n\| \recht 0$ uniformly on compact sets of $x \in G$. We must prove that $\|\mu_n - \delta_e\| \recht 0$. As in the proof of 2, we may assume that $\mu_n$ is supported on $K$ and that $\mu_n$ is $\Ad K$-invariant, so that we can construct a sequence of states $\om_n$ on $C_0(G) \rtimes^f_{\Ad} G$ such that $\om_n \recht \delta_e$ weakly$^*$, $\om_n = q \cdot \om_n \cdot q$ and $\om_n|_{C_0(G)} = \mu_n$ for all $n$.

The formula \eqref{eq.formula-cp-mult-vs-states} associates to $\om_n$ a sequence of cp-multipliers $\vphi_n$ on $\cC$ converging to $1$ pointwise. Since $\Rep K \subset \cC$ has the relative property~(T), we conclude that $\vphi_n(\pi,e) \recht 1$ uniformly on $\pi \in \Irr(K)$. Using \cite[Lemma 5.6]{PV14}, it follows that $\|\om_n|_{C_0(G)} - \delta_e\| \recht 0$. So, $\|\mu_n - \delta_e\| \recht 0$.

To prove the converse, let $\vphi_n : \Irr(\cC) \recht \C$ be a sequence of cp-multipliers on $\cC$ converging to $1$ pointwise. Denote by $\om_n$ the states on $q(C_0(G) \rtimes^f_{\Ad} G)q$ associated with $\vphi_n$ in Proposition \ref{prop.cp-multipliers-vs-states}. Since $\om_n \recht \counit$ weakly$^*$, the restriction $\mu_n := \om_n|_{C_0(G)}$ is a sequence of probability measures on $G$ such that $\mu_n \recht \delta_e$ weakly$^*$ and $\|\mu_n \circ \Ad x - \mu_n\| \recht 0$ uniformly on compact sets of $x \in G$. By our assumption, $\|\mu_n - \delta_e\| \recht 0$. For every $\pi \in \Irr(K)$, the function $\dim(\pi)^{-1} \chi_\pi$ has norm $1$. Therefore, $\om_n(\dim(\pi)^{-1} \chi_\pi) \recht 1$ uniformly on $\Irr(K)$. By \eqref{eq.formula-cp-mult-vs-states}, this means that $\vphi_n \recht 1$ uniformly on $\Irr(K)$.
\end{proof}

The following proposition gives a concrete example where $G$ has the Haagerup property, while $\cC(K < G)$ does not and even has $\Rep K$ as a full C$^*$-tensor subcategory with the relative property~(T).

\begin{proposition}\label{prop.example-SL-k-F}
Let $F$ be a non-archimedean local field with characteristic $\neq 2$. Let $k \geq 2$ and define $G = \SL(k,F)$. Let $K < G$ be a compact open subgroup, e.g.\ $K = \SL(k,\cO)$, where $\cO$ is the ring of integers of $F$. Denote by $\cC$ the rigid C$^*$-tensor category of finite rank $G$-$L^\infty(G/K)$-$L^\infty(G/K)$-modules.
\begin{enumlist}
\item $\Rep K \subset \cC$ has the relative property~(T). In particular, $\cC$ does not have the Haagerup property, although for $k=2$, the group $G$ has the Haagerup property.
\item $\cC$ has property~(T) for all $k \geq 3$.
\end{enumlist}
\end{proposition}
\begin{proof}
We denote by $\bI$ the identity element of $G = \SL(k,F)$. Let $\mu_n \in \cS(C_0(G))$ be a sequence of probability measures on $G$ such that $\mu_n \recht \delta_\bI$ weakly$^*$ and $\|\mu_n \circ \Ad x - \mu_n\| \recht 0$ uniformly on compact sets of $x \in G$. Assume that $\|\mu_n - \delta_\bI\| \not\recht 0$. Passing to a subsequence and replacing $\mu_n$ by the normalization of $\mu_n - \mu_n(\{\bI\}) \delta_\bI$, we may assume that $\mu_n(\{\bI\}) = 0$ for all $n$. Since $\mu_n \recht \delta_\bI$ weakly$^*$ and since there are at most $k$ of $k$'th roots of unity in $F$, we may also assume that $\mu_n(\{\lambda \bI\}) = 0$ for all $n$ and all $k$'th roots of unity $\lambda \in F$.

Every $\mu_n$ defines a state $\Om_n$ on the C$^*$-algebra $\cL(G)$ of all bounded Borel functions on $G$. Choose a weak$^*$-limit point $\Om \in \cL(G)^*$ of the sequence $(\Om_n)$. Then, $\Om$ induces an $\Ad G$-invariant mean on the Borel sets of $G$. In particular, $\Om$ defines an $\Ad G$-invariant mean $\Om$ on the Borel sets of the space $M_n(F)$ of $n \times n$ matrices over $F$. By Lemma \ref{lem.invariant-mean-SLn} below, $\Om$ is supported on the diagonal $F \bI \subset M_n(F)$. Since $\Om$ is also supported on $G$, it follows that $\Om$ is supported on the finite set of $\lambda \bI$ where $\lambda$ is a $k$'th root of unity in $F$. But by construction, $\Om(\{\lambda \bI\}) = 0$ for all $k$'th roots of unity $\lambda \in F$. We have reached a contradiction. So, $\|\mu_n - \delta_\bI\| \recht 0$.

By Theorem \ref{thm.amenable-haagerup-T}, $\Rep K \subset \cC$ has the relative property~(T). For $k \geq 3$, the group $\SL(k,F)$ has property~(T) and it follows from Theorem \ref{thm.amenable-haagerup-T} that $\cC$ has property~(T).
\end{proof}

The following example of \cite{dC05} illustrates that $G$ may have property~(T), while the category $\cC$ of finite rank $G$-$L^\infty(G/K)$-$L^\infty(G/K)$-modules does not.

\begin{example}
Let $F$ be a non-archimedean local field and $k \geq 3$. Define the closed subgroup $G < \SL(k+2,F)$ given by
$$G = \left\{ \begin{pmatrix} 1 & b_1 & \cdots & b_k & c \\ 0 & a_{11} & \cdots & a_{1k} & d_1 \\ \vdots & \vdots &  & \vdots & \vdots \\ 0 & a_{k1} & \cdots & a_{kk} & d_k \\ 0 & 0 & \cdots & 0 & 1\end{pmatrix} \middle| A = (a_{ij}) \in \SL(k,F) \; , \; b_i , c , d_j \in F \right\} \; .$$
As in \cite{dC05}, we get that $G$ has property~(T). Also, the center of $G$ is isomorphic with $F$ (sitting in the upper right corner) and since $F$ is non discrete, we can take a sequence $g_n \in \cZ(G)$ with $g_n \neq e$ for all $n$ and $g_n \recht e$. Using the $\Ad G$-invariant probability measures $\delta_{g_n}$, it follows from Theorem \ref{thm.amenable-haagerup-T} that $\cC$ does not have property~(T).
\end{example}

Finally, we also include a nonamenable example having the Haagerup property.

\begin{example}\label{ex.Baumslag-Solitar}
Let $2 \leq |m| < n$ be integers. Define the totally disconnected compact abelian group $K = \Z_{nm}$ as the profinite completion of $\Z$ with respect to the decreasing sequence of finite index subgroups $(n^k m^k \Z)_{k \geq 0}$. We have open subgroups $m K < K$ and $n K < K$, as well as the isomorphism $\vphi : m K \recht n K : \vphi(m k) = n k$ for all $k \in K$. We define $G$ as the HNN extension of $K$ and $\vphi$. Alternatively, we may view $K < G$ as the Schlichting completion of the Baumslag-Solitar group
$$B(m,n) = \langle a, t \mid t a^m t^{-1} = a^n \rangle$$
and the almost normal subgroup $\langle a \rangle$.

Since $G$ is acting properly on a tree, $G$ has the Haagerup property. Also, $G$ is nonamenable. For all positive integers $k,l \geq 0$, we denote by $\mu_{k,l}$ the normalized Haar measure on the open subgroup $n^k m^l K$. Note that $\vphi_*(\mu_{k,l}) = \mu_{k+1,l-1}$ whenever $k,l \geq 1$. Then the probability measures
$$\mu_n := \frac{1}{n+1} \sum_{k=0}^n \mu_{n+k,2n-k}$$
are absolutely continuous with respect to the Haar measure of $G$, and thus $c_0$ in the sense of Definition \ref{def.c-0-state}, and they satisfy $\mu_n \recht \delta_e$ weakly$^*$ and $\|\mu_n \circ \Ad x - \mu_n\| \recht 0$ uniformly on compact sets of $x \in G$. By Theorem \ref{thm.amenable-haagerup-T}, $\cC$ has the Haagerup property.
\end{example}

\begin{lemma}\label{lem.invariant-mean-SLn}
Let $F$ be a local field with characteristic $\neq 2$. Let $k \geq 2$ and define $G = \SL(k,F)$. Every $\Ad G$-invariant mean on the Borel sets of the space $M_k(F)$ of $k \times k$ matrices over $F$ is supported on the diagonal $F \bI \subset M_k(F)$.
\end{lemma}
\begin{proof}
We start by proving the lemma for $k=2$. So assume that $m$ is an $\Ad \SL(2,F)$-invariant mean on the Borel sets of $M_2(F)$.

In the proof of \cite[Proposition 1.4.12]{BHV08}, it is shown that if $m$ is a mean on the Borel sets of $F^2$ that is invariant under the transformations $\lambda \cdot (x,y) := (x + \lambda y, y)$ for all $\lambda \in F$, then
$$m(\{(x,y) \mid (x,y) \neq (0,0), |x| \leq |y| \}) = 0 \; .$$
Define $g_\lambda := \left(\begin{matrix} 1 & \lambda \\ 0 & 1 \end{matrix}\right)$ and notice that
$$g_\lambda \left(\begin{matrix} a & b \\ c & d \end{matrix}\right) g_\lambda^{-1} = \left(\begin{matrix} a + \lambda c & - \lambda a + b - \lambda^2 c + \lambda d \\ c & -\lambda c + d \end{matrix}\right) \; .$$
Hence, the map $\theta:M_2(F) \to F^2: \left(\begin{matrix} a & b \\ c & d \end{matrix}\right) \mapsto (a - d, c)$ satisfies $\theta(g_\lambda A g_\lambda^{-1}) = (2\lambda) \cdot \theta(A)$. Therefore, $m(\Om_0)=0$ for
$$\Om_0 := \left\{ \left(\begin{matrix} a & b \\ c & d \end{matrix}\right) \in M_2(F) \; \middle| \; |a-d| \leq |c| \;\text{and}\; (a-d,c) \neq (0,0) \right\} \; .$$
Taking the adjoint by $\left(\begin{matrix} \lambda & 0 \\ 0 & \lambda^{-1} \end{matrix}\right)$ for $|\lambda| \geq 2$, we get that $m(\Omega_1) = 0$ for
$$\Omega_1 := \left\{ \left(\begin{matrix} a & b \\ c & d \end{matrix}\right) \in M_2(F) \; \middle| \; |a-d| \leq 4 |c| \;\text{and}\; (a-d,c) \neq (0,0) \right\} \; .$$
For the same reason, we get that $m(\Omega_1') = 0$ for
$$\Omega_1' := \left\{ \left(\begin{matrix} a & b \\ c & d \end{matrix}\right) \in M_2(F) \; \middle| \; |a-d| \leq 4|b| \;\text{and}\; (a-d,b) \neq (0,0) \right\} \; .$$
Write $X = M_2(F) \setminus F \bI$. The matrices with $(a-d,c) = (0,0)$ belong to $\Omega_1'$ unless they are diagonal. Similarly, the matrices with $(a-d,b) = (0,0)$ belong to $\Omega_1$ unless they are diagonal. So, we find that $m(\Omega) = 0 = m(\Omega')$ for
$$\Omega = \left\{ \left(\begin{matrix} a & b \\ c & d \end{matrix}\right) \in X \; \middle| \; |a-d| \leq 4 |c| \right\} \quad\text{and}\quad
\Omega' = \left\{ \left(\begin{matrix} a & b \\ c & d \end{matrix}\right) \in X \; \middle| \; |a-d| \leq 4 |b| \right\} \; .$$
Put $\Omega\dpr := g_1 \Omega g_1^{-1}$, so that $m(\Omega\dpr) = 0$. To conclude the proof in the case $k=2$, it suffices to show that $\Omega \cup \Omega' \cup \Omega'' = X$.

Take $\left(\begin{matrix} a & b \\ c & d \end{matrix}\right) \in X \setminus (\Omega \cup \Omega')$. So, $\frac{1}{4} |a-d| > |b|, |c|$. We claim that
$$\left(\begin{matrix} a' & b' \\ c' & d' \end{matrix}\right) := g_1^{-1} \left(\begin{matrix} a & b \\ c & d \end{matrix}\right) g_1 = \left(\begin{matrix} a-c & a+b-c-d \\ c & c+d \end{matrix}\right)$$
belongs to $\Omega$. Since
$$|a' - d'| = |a-d-2c| \leq |a-d| + 2|c| < \frac{3}{2} |a -d| \quad\text{and}\quad
|b'| \geq |a-d| - |c| - |b| > \frac{1}{2}|a-d| \; ,$$
we indeed get that $|a'-d'| < 3 |b'|$. The claim follows and the lemma is proved in the case $k=2$.

For an arbitrary $k \geq 2$ and fixed $1 \leq p < q \leq k$, the map
$$M_k(F) \to M_2(F): (x_{ij}) \mapsto \left(\begin{matrix} x_{pp} & x_{pq} \\ x_{qp} & x_{qq} \end{matrix}\right)$$
is $\Ad \SL(2,F)$-equivariant. So, an $\Ad \SL(k,F)$-invariant mean $m$ on $M_k(F)$ is supported on $\{(x_{ij}) \in M_k(F) \mid x_{pp} = x_{qq}, x_{pq} = x_{qp}=0 \}$. Since $F \bI$ is the intersection of these sets, $m$ is supported on $F \bI$.
\end{proof}

\section{\boldmath Weak amenability of rigid C$^*$-tensor categories}\label{sec.weak-amen-tensor-cat}

Following \cite[Definition 5.1]{PV14}, a rigid C$^*$-tensor category is called \emph{weakly amenable} if there exists a sequence of completely bounded (cb) multipliers $\vphi_n : \Irr(\cC) \recht \C$ (see Definition \ref{def.cp-cb-multiplier}) converging to $1$ pointwise, with $\limsup_n \|\vphi_n\|\cb < \infty$ and with $\vphi_n$ finitely supported for every $n$.

Recall from the first paragraphs of Section \ref{sec.tube-algebra} the definition of the tube $*$-algebra $\cA$ of $\cC$ with respect to a full family of objects in $\cC$. To every function $\vphi : \Irr(\cC) \recht \C$, we associate the linear map
$$\theta_\vphi : \cA \recht \cA : \theta_\vphi(V) = \vphi(\al) \, V \quad\text{for all}\;\; V \in (i\al,\al j) \; .$$
We define $\|\theta_\vphi\|\cb$ by viewing $\cA$ inside its reduced C$^*$-algebra, i.e.\ by viewing $\cA \subset B(L^2(\cA,\Tr))$, where $\Tr$ is the canonical trace on $\cA$. We also consider the von Neumann algebra $\cA\dpr$ generated by $\cA$ acting on $L^2(\cA,\Tr)$.

In the following result, we clarify the link between the complete boundedness of $\vphi$ in the sense of Definition \ref{def.cp-cb-multiplier} and the complete boundedness of the map $\theta_\vphi$.

\begin{proposition}\label{prop.equality-of-cb-norms}
Let $\cC$ be a rigid C$^*$-tensor category. Denote by $\cA$ the tube $*$-algebra of $\cC$ with respect to a full family of objects in $\cC$. Let $\vphi : \Irr(\cC) \recht \C$ be any function.

Then, $\|\vphi\|\cb = \|\theta_\vphi\|\cb$. If this cb-norm is finite, we can uniquely extend $\theta_\vphi$ to a normal completely bounded map on $\cA\dpr$ having the same cb-norm.
\end{proposition}

\begin{proof}
For any family $J$ of objects, we can define the tube $*$-algebra $\cA_J$ and the linear map $\theta_\vphi^J : \cA_J \recht \cA_J$. By strong Morita equivalence, we have $\|\theta_\vphi^J\|\cb = \|\theta_\vphi\|\cb$ whenever $J$ is full and we have $\|\theta_\vphi^J\|\cb \leq \|\theta_\vphi\|\cb$ for arbitrary $J$. Also, using standard solutions for the conjugate equations, we get natural linear maps $(i \al, \al j) \recht (\overline{j} \albar, \albar \overline{i})$ and they define a trace preserving $*$-anti-isomorphism of $\cA_J$ onto $\cA_{\overline{J}}$. Defining $\vphitil : \Irr(\cC) \recht \C$ by $\vphitil(\al) = \vphi(\albar)$ for all $\al \in \Irr(\cC)$, it follows that $\|\theta_\vphi\|\cb = \|\theta_{\vphitil}\|\cb$ and it follows that $\theta_\vphi$ extends to a normal completely bounded map on $\cA\dpr$ if and only if $\theta_{\vphitil}$ extends to $\cA\dpr$.

So, it suffices to prove that $\|\vphi\|\cb = \|\theta_{\vphitil}\|\cb$ and that in the case where $\|\vphi\|\cb < \infty$, we can extend $\theta_{\vphitil}$ to a normal completely bounded map on $\cA\dpr$. First assume that $\|\theta_{\vphitil}\|\cb \leq \kappa$. Fix arbitrary objects $\al,\be \in \cC$ and write $\Psi^\vphi_{\al | \be} := \Psi^\vphi_{\al | \be,\al | \be}$. We prove that $\|\Psi^\vphi_{\al | \be}\|\cb \leq \kappa$. Since $\al,\be$ were arbitrary, it then follows that $\|\vphi\|\cb \leq \kappa$.

Note that $(\al\be,\al\be)$ is a finite dimensional C$^*$-algebra. Consider the following three bijective linear maps, making use of standard solutions of the conjugate equations.
\begin{align*}
& \eta_1 : \bigoplus_{\pi \in \Irr(\cC)} \bigl( (\al,\al \pi) \ot (\pi \be,\be)\bigr) \recht (\al \be, \al \be) : \eta_1(V \ot W) = (V \ot 1) (1 \ot W) \;\; ,\\
& \eta_2 : \bigoplus_{\pi \in \Irr(\cC)} \bigl( (\al,\al \pi) \ot (\pi \be,\be)\bigr) \recht \bigoplus_{\pi \in \Irr(\cC)} \bigl( (\al \pibar,\al) \ot (\be,\pibar \be)\bigr) : \\
& \hspace{6cm}\eta_2(V \ot W) = (V \ot 1)(1 \ot s_\pi) \ot (t_\pi^* \ot 1)(1 \ot W) \;\; ,\\
& \eta_3 : \bigoplus_{\pi \in \Irr(\cC)} \bigl( (\al \pibar,\al) \ot (\be,\pibar \be)\bigr) \recht \cA_{\be \al} : \eta_3(V \ot W) = (1 \ot V)(W \ot 1) \;\; .
\end{align*}
A direct computation shows that $\eta := \eta_3 \circ \eta_2 \circ \eta_1^{-1}$ is a unital faithful $*$-homomorphism of $(\al \be,\al \be)$ to the tube $*$-algebra $\cA_{\be \al}$. One also checks that $\theta_{\vphitil}^{\beta\al} \circ \eta = \eta \circ \Psi^\vphi_{\al | \be}$. So, we get that
$$\|\Psi^\vphi_{\al | \be}\|\cb \leq \|\theta_{\vphitil}^{\beta\al}\|\cb \leq \|\theta_{\vphitil}\|\cb \leq \kappa \; .$$

Conversely, assume that $\|\vphi\|\cb \leq \kappa$. Define the ind-objects $\rho_1$ and $\rho_2$ for $\cC$ given by
$$\rho_1 = \bigoplus_{\al,i \in \Irr(\cC)} \al i \quad\text{and}\quad \rho_2 = \bigoplus_{\al \in \Irr(\cC)} \al \; .$$
Define the type I von Neumann algebra $\cM$ of all bounded endomorphisms of $\rho_1 \rho_2$. Note that for all $\al,i,\be \in \Irr(\cC)$, we have the natural projection $p_\al \ot p_i \ot p_\be \in \cM$ and we have the identification
$$(p_\al \ot p_i \ot p_\be) \cM (p_\gamma \ot p_j \ot p_\delta) = (\al i \be, \gamma j \delta)$$
for all $\al, i , \be, \gamma , j , \delta \in \Irr(\cC)$. By our assumption, there is a normal completely bounded map $\Psi : \cM \recht \cM$ satisfying
$$\Psi(V) = \Psi^\vphi_{\al i | \be , \gamma j | \delta}(V) \quad\text{for all}\;\; V \in (\al i \be, \gamma j \delta) \; .$$
We have $\|\Psi\|\cb \leq \kappa$.

Consider the projection $q \in \cM$ given by
$$q = \sum_{\al, i \in \Irr(\cC)} p_{\albar} \ot p_i \ot p_\al \; .$$
Since $\Psi(qTq) = q \Psi(T) q$ for all $T \in \cM$, the map $\Psi$ restricts to a normal completely bounded map on $q \cM q$ with $\|\Psi|_{q \cM q}\|\cb \leq \kappa$.

Denote by $\cA$ the tube $*$-algebra associated with $\Irr(\cC)$ itself as a full family of objects. We construct a faithful normal $*$-homomorphism $\Theta : \cA\dpr \recht q \cM q$ satisfying $\Psi \circ \Theta = \Theta \circ \theta_{\vphitil}$. Once we have obtained $\Theta$, it follows that $\|\theta_{\vphitil}\|\cb \leq \kappa$ and that $\theta_{\vphitil}$ extends to a normal completely bounded map on $\cA\dpr$.

To construct $\Theta$, define the Hilbert space
$$\cH = \bigoplus_{\al,i,j \in \Irr(\cC)} (\albar i \al , j)$$
and observe that we have the natural faithful normal $*$-homomorphism $\pi : q \cM q \recht B(\cH)$ given by left multiplication. Also consider the unitary operator
$$U : L^2(\cA,\Tr) \recht \cH : U(V) = d(\al)^{-1/2} \, (1 \ot V)(t_\al \ot 1) \quad\text{for all}\;\; V \in (i \al , \al j) \; .$$
We claim that $\Theta$ can be constructed such that $\pi(\Theta(V)) = U V U^*$ for all $V \in \cA$. To prove this claim, fix $i,\al,j \in \Irr(\cC)$ and $V \in (i \al, \al j)$. For all $\gamma,\be \in \Irr(\cC)$, define the element $W_{\gamma,\be} \in (\gammabar i \gamma, \betabar j \beta)$ given by the finite sum
\begin{equation}\label{eq.Wgb}
W_{\gamma,\be} = \sum_{Z \in \onb(\gammabar \al, \betabar)} \; d(\be)^{1/2} \, d(\gamma)^{1/2} \, (1 \ot 1 \ot \Ztil) \, (1 \ot V \ot 1) \, (Z \ot 1 \ot 1) \; ,
\end{equation}
where $\Ztil = (1 \ot t_\beta^*)(1 \ot Z^* \ot 1)(s_\gamma \ot 1)$ belongs to $(\gamma, \al \be)$. A direct computation shows that
$$\langle \pi(W_{\gamma,\be}) \, U(X) , U(Y) \rangle = \langle V \cdot X , Y \rangle$$
for all $X \in (j\be, \be k)$ and $Y \in (i \gamma, \gamma l)$. So, there is a unique element $\Theta(V) \in (1 \ot p_i \ot 1)q \cM q (1 \ot p_j \ot 1)$ satisfying
$$(p_{\gammabar} \ot p_i \ot p_\gamma) \, \Theta(V) \, (p_{\betabar} \ot p_j \ot p_\beta) = W_{\gamma,\beta}$$
for all $\gamma,\beta \in \Irr(\cC)$ and $\pi(\Theta(V)) = U V U^*$.

We have defined a faithful normal $*$-homomorphism $\Theta : \cA\dpr \recht q \cM q$. It remains to prove that $\Psi \circ \Theta = \Theta \circ \theta_{\vphitil}$. Using \eqref{eq.Wgb}, it suffices to prove that
\begin{equation}\label{eq.suffice-this}
\Psi^\vphi_{\gammabar i | \al\be, \gammabar \al j | \be}(1 \ot V \ot 1) = \vphi(\albar) \, 1 \ot V \ot 1 \; .
\end{equation}
The left hand side of \eqref{eq.suffice-this} equals $1 \ot \Psi^\vphi_{i | \al, \al j | \eps}(V) \ot 1$. Writing $V = (T \ot 1)(1 \ot 1 \ot s_{\albar})$ with $T \in (i,\al j \albar)$, we have
\begin{align*}
\Psi^\vphi_{i | \al, \al j | \eps}(V) &= (T \ot 1) \, \Psi^\vphi_{\al j \albar | \al, \al j | \eps}(1 \ot 1 \ot s_{\albar}) = (T \ot 1) (1 \ot 1 \ot \Psi^\vphi_{\albar | \al, \eps | \eps}(s_{\albar}) \\ & = \vphi(\albar) \, (T \ot 1) (1 \ot 1 \ot s_{\albar}) = \vphi(\albar) \, V \; .
\end{align*}
So \eqref{eq.suffice-this} holds and the proposition is proven.
\end{proof}

\section{\boldmath Weak amenability of $\cC(K < G)$}

\begin{theorem}\label{thm.char-weak-amen}
Let $G$ be a totally disconnected group and $K < G$ a compact open subgroup. Denote by $\cC$ the rigid C$^*$-tensor category of finite rank $G$-$L^\infty(G/K)$-$L^\infty(G/K)$-modules.

Then $\cC$ is weakly amenable if and only if $G$ is weakly amenable and there exists a sequence of probability measures $\om_n \in C_0(G)^*$ that are absolutely continuous with respect to the Haar measure and such that $\om_n \recht \delta_e$ weakly$^*$ and $\|\om_n \circ \Ad x - \om_n\| \recht 0$ uniformly on compact sets of $x \in G$.

In that case, the Cowling-Haagerup constant $\Lambda(\cC)$ of $\cC$ equals $\Lambda(G)$.
\end{theorem}

In order to prove Theorem \ref{thm.char-weak-amen}, we must describe the cb-multipliers on $\cC$ in terms of completely bounded multipliers on the C$^*$-algebra $C_0(G) \rtimes^r_{\Ad} G$.

We denote by $\Pol(G)$ the $*$-algebra of locally constant, compactly supported functions on $G$. Note that $\Pol(G)$ is the linear span of the functions of the form $1_{Ly}$ where $y \in G$ and $L < G$ is a compact open subgroup. Also note that for any compact open subgroup $K_0 < G$, $\Pol(K_0)$ coincides with the $*$-algebra of coefficients of finite dimensional unitary representations of $K_0$. We define $\cE(G) = \Pol(G)^*$ as the space of all linear maps from $\Pol(G)$ to $\C$. Note that $\cE(G)$ can be identified with the space of finitely additive, complex measures on the space $\cF(G)$ of compact open subsets of $G$.

When $K_0 < G$ is a compact open subgroup, we say that a map $\mu : G \recht \cE(G)$ is $K_0$-equivariant if
$$\mu(k x k') = \mu(x) \circ \Ad k^{-1} \quad\text{for all}\;\; k,k' \in K_0 \; .$$
Note that this implies that $\mu(x)$ is $\Ad(K_0 \cap xK_0x^{-1})$-invariant for all $x \in G$.

As in \eqref{eq.identify-irr-C}, we associate to every $x \in G$ and $\pi \in \Irr(K \cap x K x^{-1})$ the irreducible object $(\pi,x) \in \Irr(\cC)$ defined as the irreducible $G$-$L^\infty(G/K)$-$L^\infty(G/K)$-module $\cH$ such that $\pi$ is isomorphic with the representation of $K \cap x K x^{-1}$ on $1_{xK} \cdot \cH \cdot 1_{eK}$. The formula
\begin{equation}\label{eq.bijection-mu-vphi}
\vphi(\pi,x) = \dim(\pi)^{-1} \, \mu(x)(\chi_\pi)
\end{equation}
then gives a bijection between the set of all functions $\vphi : \Irr(\cC) \recht \C$ and the set of all $K$-equivariant maps $\mu : G \recht \cE(G)$ with the property that $\mu(x)$ is supported on $K \cap x K x^{-1}$ for every $x \in G$.

Denote by $\cP = \Pol(L^\infty(G) \rtimes_{\Ad} G)$ the dense $*$-subalgebra defined in \eqref{eq.pol-algebra}. We always equip $\cP$ with the operator space structure inherited from $\cP \subset L^\infty(G) \rtimes_{\Ad} G$. As in Section \ref{sec.weak-amen-tensor-cat}, to every function $\vphi : \Irr(\cC) \recht \C$ is associated a linear map $\theta_\vphi : \cA \recht \cA$ on the tube $*$-algebra $\cA$ of $\cC$. We now explain how to associate to any $K_0$-equivariant map $\mu : G \recht \cE(G)$ a linear map $\Psi_\mu : \cP \recht \cP$. When $\vphi$ and $\mu$ are related by \eqref{eq.bijection-mu-vphi} and $\Theta : \cA \recht \cP$ is the $*$-anti-isomorphism of Theorem \ref{thm.our-tube-alg}, it will turn out that $\Psi_\mu \circ \Theta = \Theta \circ \theta_\vphi$, so that in particular, $\|\theta_\vphi\|\cb = \|\Psi_\mu\|\cb$. We will further prove a criterion for $\Psi_\mu$ to be completely bounded and that will be the main tool to prove Theorem \ref{thm.char-weak-amen}.

Denote $\Delta : L^\infty(G) \recht L^\infty(G \times G) : \Delta(F)(g,h) = F(gh)$. For every $\mu \in \cE(G)$, the linear map
$$\psi_\mu : \Pol(G) \recht \Pol(G) : \psi_\mu(F) = (\mu \ot \id)\Delta(F)$$
is well defined. When $\mu : G \recht \cE(G)$ is $K_0$-equivariant with respect to the compact open subgroup $K_0 < G$, we define
$$\Psi_\mu : \cP \recht \cP : \Psi_\mu(F \, u_x \, p_L) = \psi_{\mu(x)}(F) \, u_x \, p_L$$
for every $F \in \Pol(G)$, $x \in G$ and open subgroup $L < K_0$.

\begin{lemma} \label{lem.compatible-theta-Psi}
Denote by $\Theta : \cA \recht \cP$ the $*$-anti-isomorphism constructed in Theorem \ref{thm.our-tube-alg} between the tube $*$-algebra $\cA$ and $\cP = \Pol(L^\infty(G) \rtimes_{\Ad} G)$. Let $\vphi : \Irr(\cC) \recht \C$ be any function and denote by $\mu : G \recht \cE(G)$ the associated $K$-equivariant map given by \eqref{eq.bijection-mu-vphi} with $\mu(x)$ supported in $K \cap xKx^{-1}$ for all $x \in G$. Then, $\Psi_\mu \circ \Theta = \Theta \circ \theta_\vphi$.
\end{lemma}
\begin{proof}
The result follows from a direct computation using \eqref{eq.concrete-formula}.
\end{proof}

We prove the following technical result in exactly the same way as \cite{Jo91}.

\begin{lemma} \label{lem.char-Psi-mu-cb}
Let $K_0,K < G$ be compact open subgroups and $\mu : G \recht \cE(G)$ a $K_0$-equivariant map. Let $\kappa \geq 0$. Then the following conditions are equivalent.
\begin{enumlist}
\item\label{one} $\Psi_\mu$ extends to a completely bounded map on $C_0(G) \rtimes^r_{\Ad} G$ with $\|\Psi_\mu\|\cb \leq \kappa$.
\item\label{two} $\Psi_\mu$ extends to a normal completely bounded map on $L^\infty(G) \rtimes_{\Ad} G$ with $\|\Psi_\mu\|\cb \leq \kappa$.
\item\label{three} There exists a nondegenerate $*$-representation $\pi : C_0(G) \rtimes^f_{\Ad} G \recht B(\cK)$ and bounded maps $V, W : G \recht \cK$ such that
\begin{itemlist}
\item $V(k x k') = \pi(k) V(x)$ and $W(k x k') = \pi(k) W(x)$ for all $x \in G$, $k \in K_0$ and $k' \in K$,
\item $\mu(zy^{-1})(F) = \langle \pi(F) \pi(zy^{-1}) V(y) , W(z) \rangle$ for all $F \in \Pol(G)$ and $y,z \in G$,
\item $\|V\|_\infty \, \|W\|_\infty \leq \kappa$.
\end{itemlist}
In particular, every $\mu(x)$ is an actual complex measure on $G$, i.e.\ $\mu(x) \in C_0(G)^*$.
\end{enumlist}
\end{lemma}
\begin{proof}
\ref{one} $\Rightarrow$ \ref{three}. Denote $P = C_0(G) \rtimes^r_{\Ad} G$ and consider the (unique) completely bounded extension of $\Psi_\mu$ to $P$, which we still denote as $\Psi_\mu$. Define the nondegenerate $*$-representation
$$\zeta : P \recht B(L^2(G)) : \zeta(F) = F(e) 1 \;\;\text{and}\;\; \zeta(u_x) = \lambda_x$$
for all $F \in C_0(G)$, $x \in G$. Then $\zeta \circ \Psi_\mu : P \recht B(L^2(G))$ has cb norm bounded by $\kappa$ and satisfies
$$(\zeta \circ \Psi_\mu)(u_k S u_{k'}) = \lambda_k \, (\zeta \circ \Psi_\mu)(S) \, \lambda_{k'}$$
for all $S \in P$, $k,k' \in K_0$. By the Stinespring dilation theorem proved in \cite[Theorem B.7]{BO08}, we can choose a nondegenerate $*$-representation $\pi : P \recht B(\cK)$ and bounded operators $\cV, \cW : L^2(G) \recht \cK$ such that
\begin{itemlist}
\item $(\zeta \circ \Psi_\mu)(S) = \cW^* \pi(S) \cV$ for all $S \in P$,
\item $\cV \lambda_k = \pi(k) \cV$ and $\cW \lambda_k = \pi(k) \cW$ for all $k \in K_0$,
\item $\|\cV\| \, \|\cW\| = \|\Psi_\mu\|\cb \leq \kappa$.
\end{itemlist}
We normalize the left Haar measure on $G$ such that $\lambda(K)=1$ and define the maps $V, W : G \recht \cK$ given by $V(y) = \cV(1_{yK})$ and $W(z) = \cW(1_{zK})$. By construction, \ref{three} holds.

\ref{three} $\Rightarrow$ \ref{two}. Write $P\dpr = L^\infty(G) \rtimes_{\Ad} G$. Denote by $\pi_r : P\dpr \recht B(L^2(G \times G))$ the standard representation given by
$$(\pi_r(F) \xi)(g,h) = F(hgh^{-1}) \xi(g,h) \quad\text{and}\quad (\pi_r(u_x)\xi)(g,h) = \xi(g,x^{-1}h)$$
for all $g,h,x \in G$, $F \in L^\infty(G)$. For every nondegenerate $*$-representation $\pi : C_0(G) \rtimes^f_{\Ad} G \recht B(\cK)$, there is a unique normal $*$-homomorphism $\pitil : P\dpr \recht B(\cK \ot L^2(G \times G))$ satisfying
$$\pitil(F) = (\pi \ot \pi_r)\Delta(F) \quad\text{and}\quad \pitil(u_x) = \pi(x) \ot \pi_r(x)$$
for all $F \in C_0(G)$, $x \in G$. Given $V$ and $W$ as in \ref{three}, we then define the bounded operators $\cV, \cW : L^2(G \times G) \recht \cK \ot L^2(G \times G)$ by
$$(\cV \xi)(g,h) = \xi(g,h) V(h) \quad\text{and}\quad (\cW \xi)(g,h) = \xi(g,h) W(h)$$
for all $g,h \in G$. Note that $\|\cV\| = \|V\|_\infty$ and $\|\cW\| = \|W\|_\infty$. Since $\Psi_\mu(T) = \cW^* \pitil(T) \cV$ for all $T \in \cP$, it follows that \ref{two} holds.

\ref{two} $\Rightarrow$ \ref{one} is trivial.
\end{proof}

We are now ready to prove Theorem \ref{thm.char-weak-amen}. We follow closely the proof of \cite[Theorem A]{Oz10}.

\begin{proof}[Proof of Theorem \ref{thm.char-weak-amen}]
We define $\cQ(G)$ as the set of all maps $\mu : G \recht \cE(G)$ satisfying the following properties:
\begin{itemlist}
\item there exists a compact open subgroup $K_0 < G$ such that $\mu$ is $K_0$-equivariant,
\item for every $x \in G$, we have that $\mu(x) \in C_0(G)^*$, $\mu(x)$ is compactly supported and $\mu(x)$ is absolutely continuous with respect to the Haar measure,
\item $\|\Psi_\mu\|\cb < \infty$.
\end{itemlist}
Writing $\|\mu\|\cb := \|\Psi_\mu\|\cb$, we call a sequence $\mu_n \in \cQ(G)$ a cbai (completely bounded approximate identity) if
\begin{itemlist}
\item $\limsup_n \|\mu_n\|\cb < \infty$,
\item for every $F \in C_0(G)$, we have that $\mu_n(x)(F) \recht F(e)$ uniformly on compact sets of $x \in G$,
\item for every $n$, we have that $\mu_n$ has compact support (i.e. $\mu_n(x) = 0$ for all $x$ outside a compact subset of $G$).
\end{itemlist}
If a cbai exists, we define $\Gamma(G)$ as the smallest possible value of $\limsup_n \|\mu_n\|\cb$, where $(\mu_n)$ runs over all cbai. Note that this smallest possible value is always attained by a cbai.

First assume that $\cC$ is weakly amenable. By Proposition \ref{prop.equality-of-cb-norms}, we can take a sequence of finitely supported functions $\vphi_n : \Irr(\cC) \recht \C$ converging to $1$ pointwise and satisfying $\limsup_n \|\theta_{\vphi_n}\|\cb = \Lambda(\cC)$ where $\theta_{\vphi_n} : \cA \recht \cA$ as before. Define the $K$-equivariant maps $\mu_n : G \recht \cE(G)$ associated with $\vphi_n$ by \eqref{eq.bijection-mu-vphi}.

For a fixed $n$ and a fixed $x \in G$, there are only finitely many $\pi \in \Irr(K \cap x K x^{-1})$ such that $\vphi_n(\pi,x) \neq 0$. So, $\mu_n(x)$ is an actual complex measure on $K \cap xKx^{-1}$ that is absolutely continuous with respect to the Haar measure (and with the Radon-Nikodym derivative being in $\Pol(K \cap x K x^{-1})$). By Lemma \ref{lem.compatible-theta-Psi}, $\|\Psi_{\mu_n}\|\cb = \|\theta_{\vphi_n}\|\cb < \infty$. So, $\mu_n \in \cQ(G)$ and the sequence $(\mu_n)$ is a cbai with $\limsup_n \|\mu_n\|\cb \leq \Lambda(\cC)$. Thus, $\Gamma(G) \leq \Lambda(\cC)$. Write $\kappa = \Gamma(G)^{1/2}$.

For every map $\mu : G \recht \cE(G)$, we define
$$\mubar : G \recht \cE(G) : \mubar(x)(F) = \overline{(\mu(x^{-1}) \circ \Ad (x^{-1}))(\overline{F})} \; .$$
If $\mu$ is $K_0$-equivariant, also $\mubar$ is $K_0$-equivariant and $\Psi_{\mubar}(T) = (\Psi_\mu(T^*))^*$ for all $T \in \cP$. So, $\|\mubar\|\cb = \|\mu\|\cb$. Also, if $(\mu_n)$ is a cbai, then $(\overline{\mu_n})$ is a cbai.

Since $\Gamma(G) = \kappa^2 < \infty$ and using Lemma \ref{lem.char-Psi-mu-cb}, we can take a cbai $(\mu_n)$, a nondegenerate $*$-representation $\pi : C_0(G) \rtimes^f_{\Ad} G \recht B(\cK)$ and bounded functions $V_n , W_n : G \recht \cK$ as in Lemma \ref{lem.char-Psi-mu-cb}.\ref{three} with
$$\lim_n \|V_n\|_\infty = \kappa = \lim_n \|W_n\|_\infty \; .$$
Replacing $\mu_n$ by $(\mu_n + \overline{\mu_n})/2$, we may assume that $\mu_n = \overline{\mu_n}$ for all $n$. It then follows that both formulas
\begin{align*}
& \mu_n(z y^{-1})(F) = \langle \pi(F) \pi(zy^{-1}) V_n(y) , W_n(z) \rangle \quad\text{and}\\
& \mu_n(z y^{-1})(F) = \langle \pi(F) \pi(zy^{-1}) W_n(y) , V_n(z) \rangle
\end{align*}
hold for all $F \in C_0(G)$ and $y,z \in G$.

Put $\eta_n := \mu_n(e)$. We prove that $\|\eta_n \circ \Ad x - \eta_n\| \recht 0$ uniformly on compact sets of $x \in G$. To prove this statement, fix an arbitrary compact subset $C \subset G$ and an arbitrary sequence $x_n \in C$. Define
$$\zeta_n : G \recht \cE(G) : \zeta_n(x) = \mu_n(x_n x) \circ \Ad x_n \; .$$
Since $\Psi_{\zeta_n}(T) = u_{x_n}^* \Psi_{\mu_n}(u_{x_n} T)$, it follows that $(\zeta_n)$ is a cbai. Also note that for all $y,z \in G$ and $F \in C_0(G)$, we have
$$\zeta_n(zy^{-1})(F) = \langle \pi((\Ad x_n)(F)) \, \pi(x_n z y^{-1}) \, V_n(y) , W_n(x_n z) \rangle
= \langle \pi(F) \pi(zy^{-1}) V_n(y), W_n'(z)\rangle \; ,$$
with $W_n'(z) = \pi(x_n)^* W_n(x_n z)$. Then also $(\mu_n + \zeta_n)/2$ is a cbai satisfying
$$\frac{1}{2}(\mu_n + \zeta_n)(zy^{-1})(F) = \langle \pi(F) \pi(zy^{-1}) V_n(y) , (W_n(z) + W_n'(z))/2 \rangle$$
for all $y,z \in G$ and $F \in C_0(G)$. We conclude that
\begin{align*}
\kappa^2 & \leq \liminf_n \|V_n\|_\infty \; \|(W_n + W_n')/2\|_\infty = \kappa \; \liminf_n \|(W_n + W_n')/2\|_\infty \\
& \leq \kappa \; \limsup_n \|(W_n + W_n')/2\|_\infty \leq \kappa \; \frac{1}{2} \limsup_n \bigl(\|W_n\|_\infty + \|W_n'\|_\infty\bigr) = \kappa^2 \; .
\end{align*}
Therefore, $\lim_n \|(W_n + W_n')/2\|_\infty = \kappa$. So, we can choose $z_n \in G$ such that $\lim_n \|(W_n(z_n) + W_n'(z_n))/2\| = \kappa$. Since also
$\limsup_n \|W_n(z_n)\| \leq \kappa$ and $\limsup_n \|W_n'(z_n)\| \leq \kappa$, the parallelogram law implies that $\lim_n \|W_n(z_n) - W_n'(z_n)\| = 0$.

Since for all $F \in C_0(G)$,
\begin{align*}
& \zeta_n(e)(F) = \zeta_n (z_n z_n^{-1})(F) = \langle \pi(F) V_n(z_n) , W_n'(z_n) \rangle \quad\text{and} \\
& \mu_n(e)(F) = \mu_n(z_n z_n^{-1})(F) = \langle \pi(F) V_n(z_n), W_n(z_n) \rangle \; ,
\end{align*}
it follows that $\lim_n \|\zeta_n(e) - \mu_n(e)\| = 0$. This means that $\lim_n \|\mu_n(x_n) \circ \Ad x_n - \mu_n(e)\| = 0$. Since the sequence $x_n \in C$ was arbitrary, we have proved that $\lim_n \|\mu_n(x) - \mu_n(e) \circ \Ad x^{-1} \| = 0$ uniformly on compact sets of $x \in G$.

Reasoning in a similar way with $\zeta_n : G \recht \cE(G) : \zeta_n(x) = \mu_n(x x_n^{-1})$, which satisfies
$$\zeta_n(zy^{-1})(F) = \langle \pi(F) \pi(zy^{-1}) V_n'(y) , W_n(z) \rangle$$
with $V_n'(y) = \pi(x_n)^* V_n(x_n y)$, we also find that $\lim_n \|\mu_n(x) - \mu_n(e)\| = 0$ uniformly on compact sets of $x \in G$. Both statements together imply that $\|\eta_n \circ \Ad x - \eta_n\| \recht 0$ uniformly on compact sets of $x \in G$.

We next claim that for every $H \in \Pol(G)$ with $H(e) = 1$ and $\|H\|_\infty = 1$, we have that $\lim_n \|\eta_n \cdot H - \eta_n\| = 0$. To prove this claim, define
$$\zeta_n : G \recht \cE(G) : \zeta_n(x)(F) = \mu_n(x)(H F) \; .$$
Since $\zeta_n(zy^{-1})(F) = \langle \pi(F) \pi(zy^{-1}) V_n(y) , W_n'(z)\rangle$ with $W_n'(z) = \pi(H)^* W_n(z)$ and because the function $H \in \Pol(G)$ is both left and right $K_0$-invariant for a small enough compact open subgroup $K_0 < G$, it follows from Lemma \ref{lem.char-Psi-mu-cb} that
$$\|\zeta_n\|\cb \leq \|V_n\|_\infty \, \|W_n'\|_\infty \leq \|V_n\|_\infty \, \|W_n\|_\infty = \|\mu_n\|\cb \; .$$
So again, $(\zeta_n)$ and $(\mu_n + \zeta_n)/2$ are cbai. The same reasoning as above gives us a sequence $z_n \in G$ with $\lim_n \|W_n(z_n) - W_n'(z_n) \| = 0$, which allows us to conclude that $\lim_n \|\mu_n(e) - \zeta_n(e)\| = 0$, thus proving the claim.

Altogether, we have proved that $\eta_n \in C_0(G)^*$ is a sequence of complex measures that are absolutely continuous with respect to the Haar measure and that satisfy
\begin{itemlist}
\item $\|\eta_n - \eta_n \circ \Ad x \| \recht 0$ uniformly on compact sets of $x \in G$,
\item $\|\eta_n \cdot 1_L - \eta_n \| \recht 0$ for every compact open subset $L \subset G$ with $e \in L$,
\item $\eta_n(F) \recht F(e)$ for every $F \in C_0(G)$.
\end{itemlist}
In particular, $\liminf_n \|\eta_n\| \geq 1$. But then $\om_n := \|\eta_n\|^{-1} \, |\eta_n|$ is a sequence of probability measures on $G$ that are absolutely continuous with respect to the Haar measure and satisfy $\om_n \recht \delta_e$ weakly$^*$ and $\|\om_n \circ \Ad x - \om_n\| \recht 0$ uniformly on compact sets of $x \in G$.

By Lemma \ref{lem.char-Psi-mu-cb}, the maps $\Psi_{\mu_n}$ extend to normal cb maps on $L^\infty(G) \rtimes_{\Ad} G$. Restricting to $L(G)$, we obtain the compactly supported Herz-Schur multipliers
$$L(G) \recht L(G) : u_x \mapsto \gamma_n(x) u_x \quad\text{for all}\;\; x \in G \; ,$$
where $\gamma_n : G \recht \C$ is the compactly supported, locally constant function given by $\gamma_n(x) = \mu_n(x)(1)$. So, $G$ is weakly amenable and
$$\Lambda(G) \leq \limsup_n \|\Psi_{\mu_n}|_{L(G)}\|\cb \leq \limsup_n \|\Psi_{\mu_n}\| \leq \Lambda(\cC) \; .$$

Conversely, assume that $G$ is weakly amenable and that there exists a sequence of probability measures $\om_n \in C_0(G)^*$ that are absolutely continuous with respect to the Haar measure and such that $\om_n \recht \delta_e$ weakly$^*$ and $\|\om_n \circ \Ad x - \om_n\| \recht 0$ uniformly on compact sets of $x \in G$.

Since $G$ is weakly amenable, we can take a sequence of $K$-biinvariant Herz-Schur multipliers $\zeta_n : G \recht \C$ having compact support, converging to $1$ uniformly on compacta and satisfying $\limsup_n \|\zeta_n\|\cb = \Lambda(G)$.

Denote by $\Pol(G)^+$ the set of positive, locally constant, compactly supported functions on $G$. Denote by $h \in C_0(G)^*$ the Haar measure on the compact open subgroup $K < G$. Approximating $\om_n$, we may assume that $\om_n = h \cdot \xi_n^2$, where $\xi_n$ is a sequence of $\Ad K$-invariant functions in $\Pol(K)^+$. Define the representation $\pi : C_0(G) \rtimes^f_{\Ad} G \recht B(L^2(G))$ given by
$$(\pi(F) \xi)(g) = F(g) \xi(g) \quad\text{and}\quad (\pi(x) \xi)(g) = \Delta(x)^{1/2} \, \xi(x^{-1} g x)$$
for all $F \in C_0(G)$, $\xi \in L^2(G)$ and $x,g \in G$. We then define the $K$-equivariant map
$$\mu_n : G \recht C_0(G)^* : \mu_n(x)(F) = \zeta_n(x) \, \langle \pi(F) \pi(x) \xi_n , \xi_n \rangle \; .$$
Since $\xi_n$ is an $\Ad K$-invariant element of $\Pol(K)$ and $\pi(x) \xi_n$ is an $\Ad (x K x^{-1})$-invariant element of $\Pol(x K x^{-1})$, we get that $\mu_n(x)$ is an $\Ad(K \cap x K x^{-1})$-invariant complex measure supported on $K \cap x K x^{-1}$ and having a density in $\Pol(K \cap x K x^{-1})$ with respect to the Haar measure. Since moreover $\zeta_n$ is compactly supported, it follows that the functions $\vphi_n : \Irr(\cC) \recht \C$ associated with $\mu_n$ through \eqref{eq.bijection-mu-vphi} are finitely supported.

Since $\|\om_n \circ \Ad x - \om_n\| \recht 0$ for every $x \in G$, we have that $\|\pi(x) \xi_n - \xi_n\| \recht 0$ for every $x \in G$. Since $\om_n \recht \delta_e$ weakly$^*$, we have that $\langle \pi(F) \xi_n , \xi_n \rangle \recht F(e)$ for every $F \in C_0(G)$. Both together imply that $\vphi_n \recht 1$ pointwise.

To conclude the proof of the theorem, by Lemma \ref{lem.compatible-theta-Psi}, it suffices to prove that $\limsup_n \|\mu_n\|\cb \leq \Lambda(G)$.

Since $\zeta_n$ is a $K$-biinvariant Herz-Schur multiplier on $G$, we can choose a Hilbert space $\cK$ and $K$-biinvariant functions $V_n , W_n : G \recht \cK$ such that
\begin{equation}\label{eq.dec-on-G}
\|V_n\|_\infty \, \|W_n\|_\infty = \|\zeta_n\|\cb \quad\text{and}\quad \zeta_n(z y^{-1}) = \langle V_n(y) , W_n(z) \rangle
\end{equation}
for all $y,z \in G$. We equip $L^2(G) \ot \cK$ with the $*$-representation of $C_0(G) \rtimes^f_{\Ad} G$ given by $\pi(\,\cdot\,) \ot 1$. We define the bounded maps
$$\cV_n : G \recht L^2(G) \ot \cK : \cV_n(y) = \xi_n \ot V_n(y) \quad\text{and}\quad \cW_n : G \recht L^2(G) \ot \cK : \cW_n(y) = \xi_n \ot W_n(y) \; .$$
One checks that
$$\mu_n(zy^{-1})(F) = \langle (\pi(F) \pi(zy^{-1}) \ot 1) \cV_n(y) , \cW_n(z) \rangle$$
for all $y,z$ and that all other conditions in Lemma \ref{lem.char-Psi-mu-cb}.\ref{three} are satisfied, with $\|\cV_n\|_\infty = \|V_n\|_\infty$ and $\|\cW_n\|_\infty = \|W_n\|_\infty$. So, we conclude that
$$\limsup_n \|\mu_n\|\cb \leq \limsup_n \|\zeta_n\|\cb = \Lambda(G)$$
and this ends the proof of the theorem.
\end{proof}

\begin{example}
Taking $G$ as in Example \ref{ex.Baumslag-Solitar}, the category $\cC$ is weakly amenable with $\Lambda(\cC) = 1$. Indeed, $G$ is weakly amenable with $\Lambda(G) = 1$ and the probability measures $\mu_n$ constructed in Example \ref{ex.Baumslag-Solitar} are absolutely continuous with respect to the Haar measure, so that the result follows from Theorem \ref{thm.char-weak-amen}.

Taking $G = \SL(2,F)$ as in Proposition \ref{prop.example-SL-k-F}, we get that $\cC$ is not weakly amenable, although $G$ is weakly amenable with $\Lambda(G) = 1$.
\end{example}

\end{document}